\documentclass[11pt,a4paper]{amsart}
\usepackage{amssymb,amscd}
\usepackage{mathrsfs}
\usepackage[mathcal]{eucal}
\usepackage{float}

\theoremstyle{plain}
\newtheorem{thm}{Theorem}[section]
\newtheorem{lem}[thm]{Lemma}
\newtheorem{pro}[thm]{Proposition}
\newtheorem{cor}[thm]{Corollary}

\newtheorem*{con*}{Conjecture}

\theoremstyle{remark}\theoremstyle{remark}
\newtheorem*{rem*}{Remark}
\newtheorem*{qun*}{Question}

\newtheorem{con}[thm]{Conjecture}

\newtheorem*{acknowledgements}{Acknowledgements}

\numberwithin{equation}{section}
\numberwithin{table}{section}

\newcommand{\N}{\mathbb{N}}
\newcommand{\Z}{\mathbb{Z}}
\newcommand{\Q}{\mathbb{Q}}
\newcommand{\F}{\mathbb{F}}
\newcommand{\C}{\mathbb{C}}

\newcommand{\R}{\mathbb{R}}

\newcommand{\spl}{\mathfrak{sl}}
\newcommand{\gl}{\mathfrak{gl}}

\newcommand{\mfb}{\mathfrak{b}}
\newcommand{\mfg}{\mathfrak{g}}
\newcommand{\mfn}{\mathfrak{n}}
\newcommand{\mfh}{\mathfrak{h}}
\newcommand{\lat}{\Lambda}

\newcommand{\mfp}{\mathfrak{p}}
\newcommand{\mfP}{\mathfrak{P}}

\newcommand{\lri}{\mathfrak{o}}
\newcommand{\Lri}{\mathfrak{O}}
\newcommand{\lfi}{\mathfrak{k}}
\newcommand{\Lfi}{\mathfrak{K}}
\newcommand{\ldr}{\mathfrak{R}}
\newcommand{\Ldr}{\mathfrak{D}}

\newcommand{\caseA}{A}
\newcommand{\caseB}{C}
\newcommand{\caseC}{B}
\newcommand{\caseD}{E}
\newcommand{\caseE}{F}
\newcommand{\caseF}{D}
\newcommand{\caseG}{G}
\newcommand{\caseH}{K}
\newcommand{\caseI}{H}
\newcommand{\caseJ}{I}
\newcommand{\caseK}{J}

\newcommand{\Anzahl}{a}
\newcommand{\Mass}{m}
\newcommand{\Menge}{A}

\newcommand{\gri}{\ensuremath{{\scriptstyle \mathcal{O}}}}
\newcommand{\smallgri}{\ensuremath{{\scriptscriptstyle \mathcal{O}}}}

\renewcommand{\epsilon}{\varepsilon}
\renewcommand{\phi}{\varphi}
\renewcommand{\theta}{\vartheta}

\DeclareMathOperator{\ad}{ad}

\DeclareMathOperator{\Irr}{Irr}
\DeclareMathOperator{\Rad}{Rad}

\DeclareMathOperator{\rk}{rk}
\DeclareMathOperator{\SL}{SL}
\DeclareMathOperator{\GL}{GL}
\DeclareMathOperator{\PGL}{PGL}
\DeclareMathOperator{\Aut}{Aut}
\DeclareMathOperator{\Hom}{Hom}
\DeclareMathOperator{\Mat}{Mat}

\DeclareMathOperator{\res}{res}
\DeclareMathOperator{\Cen}{C}

\DeclareMathOperator{\Tr}{Tr}

\DeclareMathOperator{\real}{Re}

\DeclareMathOperator{\Id}{Id}

\newcommand{\dotcup}{\ensuremath{\mathaccent\cdot\cup}}
\newcommand{\bigdotcup}{\ensuremath{\mathop{\dot{\bigcup}}}}
\def \bfa {{\bf a}}
\def \bfb {{\bf b}}

\def \bfy {{\bf y}}

\def \mcR {\ensuremath{\mathcal{R}}}

\def \mcZ {\ensuremath{\mathcal{Z}}}

\def \Zp  {\mathbb{Z}_p}

\newcommand{\Qp}{\mathbb{Q}_p}

\begin{document}
\title[Representation zeta functions of some $p$-adic analytic
  groups]{Representation zeta functions of some compact $p$-adic
  analytic groups}

%

\author{Nir Avni}\thanks{Avni was supported by NSF grant DMS-0901638.}
\address{Department of Mathematics, Harvard University, One Oxford
  Street, Cambridge MA 02138, USA} \email{avni.nir@gmail.com}

\author{Benjamin Klopsch} \address{Department of Mathematics, Royal
  Holloway, University of London, Egham TW20 0EX, United Kingdom}
\email{Benjamin.Klopsch@rhul.ac.uk}

\author{Uri Onn} \address{Department of Mathematics, Ben Gurion
  University of the Negev, Beer-Sheva 84105 Israel}
  \email{urionn@math.bgu.ac.il}

\author{Christopher Voll} \address{School of Mathematics, University
  of Southampton, University Road, Southampton SO17 1BJ, United
  Kingdom}
\email{C.Voll.98@cantab.net}

\begin{abstract}
  Using the Kirillov orbit method, novel methods from $\mfp$-adic
  integration and Clifford theory, we study representation zeta
  functions associated to compact $p$\nobreakdash-adic analytic
  groups.  In particular, we give general estimates for the abscissae
  of convergence of such zeta functions.

  We compute explicit formulae for the representation zeta functions
  of some compact $p$-adic analytic groups, defined over a compact
  discrete valuation ring $\lri$ of characteristic~$0$.  These include
  principal congruence subgroups of $\SL_2(\lri)$, without any
  restrictions on the residue field characteristic of $\lri$, as well
  as the norm one group $\SL_1(\Ldr)$ of a non-split quaternion
  algebra $\Ldr$ over the field of fractions of $\lri$ and its
  principal congruence subgroups.  We also determine the
  representation zeta functions of principal congruence subgroups of
  $\SL_3(\lri)$ in the case that $\lri$ has residue field
  characteristic $3$ and is unramified over $\Z_3$.
\end{abstract}




\maketitle



\section{Introduction}\label{sec:introduction}

Let $G$ be a group.  For $n\in\N$, we denote by $r_n(G)$ the number of
isomorphism classes of $n$-dimensional irreducible complex
representations of~$G$.  If $G$ is a topological group, we tacitly
assume that representations are continuous. We call $G$
(representation) \emph{rigid} if, for every $n\in\N$, the number
$r_n(G)$ is finite.  It is known that a profinite group is rigid if
and only if it is FAb, i.e.\ if each of its open subgroups has finite
abelianisation.  All the groups studied in this paper are rigid.

We say that $G$ has \emph{polynomial representation growth} if the
sequence $R_N(G):=\sum_{n=1}^N r_n(G)$ is bounded by a polynomial in
$N$. The \emph{representation zeta function} of a group $G$ with
polynomial representation growth is the Dirichlet series
$$
\zeta_G(s) := \sum_{n=1}^\infty r_n(G) n^{-s},
$$ 
where $s$ is a complex variable.  The \emph{abscissa of convergence}
$\alpha(G)$ of $\zeta_G(s)$, i.e.\ the infimum of all $\alpha\in\R$
such that $\zeta_G(s)$ converges on the complex right half-plane $\{s
\in \C \mid \real (s)> \alpha \}$, gives the precise degree of
polynomial growth: $\alpha(G)$ is the smallest $\alpha$ such that
$R_n(G) = O(1+n^{\alpha+\epsilon})$ for every $\epsilon \in \R_{>0}$.
Representation zeta functions of groups have been studied in a variety
of contexts, e.g.\ in the setting of arithmetic and $p$-adic analytic
groups, wreath products of finite groups and finitely generated
nilpotent groups; cf.\ \cite{Ja06, LaLu08, AvKlOnVo10, AvKlOnVo10+,
  AvKlOnVo09}, \cite{BadlHa10} and \cite{HrMa07, Vo10}. In the case of
finitely generated nilpotent groups, one enumerates representations up
to twisting by one-dimensional representations.

\smallskip

We state our main results, deferring the precise definitions of some
of the technical terms involved as well as further corollaries and
remarks to later sections.  In our first theorem we provide general
estimates for the abscissae of convergence of representation zeta
functions of FAb compact $p$-adic analytic groups.  These are
expressed in terms of invariants $\rho$ and $\sigma$ which are defined
by minimal and maximal centraliser dimensions in the corresponding Lie
algebras; see~\eqref{equ:sigmarho_centraliser}. The term `permissible'
is defined in Section~\ref{subsec:basic_set_up}; here it suffices to
note that, for a given Lie lattice $\mfg$, all sufficiently large
$m\in\N$ are permissible for~$\mfg$.
 
\begin{thm}\label{thm:abscissa_estimate}
  Let $\lri$ be a compact discrete valuation ring of characteristic
  $0$, with maximal ideal $\mfp$ and field of fractions $\lfi$.  Let
  $\mfg$ be an $\lri$-Lie lattice such that $\lfi \otimes_\lri \mfg$
  is a perfect Lie algebra.  Let $d := \dim_\lfi (\lfi \otimes_\lri
  \mfg)$, and let $\sigma := \sigma(\mfg)$ and $\rho := \rho(\mfg)$, as
  defined in~\eqref{equ:sigmarho_centraliser}.  Let $m \in \N$ be
  permissible for $\mfg$, and let $\mathsf{G}^m := \exp(\mfp^m \mfg)$.

  Then lower and upper bounds for the abscissa of convergence of
  $\zeta_{\mathsf{G}^m}(s)$ are given by
  $$
  (d - 2\rho) \rho^{-1} \leq \alpha(\mathsf{G}^m) \leq (d-2\sigma)
  \sigma^{-1}.
  $$
\end{thm}

Our other main results provide explicit formulae for the
representation zeta functions of members of specific families of
`simple' compact $p$-adic analytic groups.  The groups are defined
over a compact discrete valuation ring $\lri$ of characteristic $0$
and residue field characteristic $p$, with absolute ramification index
$e(\lri,\Zp)$. In the unramified case $e(\lri,\Zp)=1$ all $m \in \N$
are permissible for a given $\lri$-Lie lattice~$\mfg$.

\begin{thm}\label{thm:SL2}
  Let $\lri$ be a compact discrete valuation ring of characteristic
  $0$.  Then for all $m \in \N$ which are permissible for the Lie
  lattice $\spl_2(\lri)$ the representation zeta function of the
  principal congruence subgroup $\SL_2^m(\lri)$ is
  $$
  \zeta_{\SL_2^m(\lri)}(s) =
  \begin{cases}
    q^{3m} \frac{1 - q^{-2-s}}{1 - q^{1-s}} & \text{if $p>2$,} \\
    q^{3m} \frac{q^2 - q^{-s}}{1 - q^{1-s}} & \text{if $p=2$ and
      $e(\lri,\Z_2) = 1$}.
  \end{cases}
  $$
\end{thm}

\begin{thm} \label{thm:quaternions} Let $\Ldr$ be a non-split
  quaternion algebra over a $\mfp$-adic field $\lfi$. Let $\ldr$
  denote the maximal compact subring of $\Ldr$, and suppose that $2
  e(\lfi,\Q_p) < p-1$, where $p$ denotes the residue field
  characteristic of $\lfi$.

  Then the representation zeta functions of $\SL_1(\Ldr) =
  \SL_1(\ldr)$ and its principal congruence subgroups $\SL_1^m(\ldr)$,
  $m \in \N$, are
  \begin{align*}
    \zeta_{\SL_1(\ldr)}(s) & = \frac{(q+1) (1 -
      q^{-s}) + 4 (q-1) ((q+1)/2)^{-s}}{1-q^{1-s}}, \\
    \zeta_{\SL_1^m(\ldr)}(s) & = q^{3m} \frac{q -
      q^{-1-s}}{1-q^{1-s}}.
  \end{align*}
\end{thm}

\begin{thm}
  \label{thm:SL3_p=3}
  Let $\lri$ be a compact discrete valuation ring of
  characteristic~$0$, with residue field of cardinality~$q$ and
  characteristic $3$.  Suppose that $\lri$ is unramified over $\Z_3$.
  Then, for all $m \in \N$, one has
  \begin{equation*}
    \zeta_{\SL_3^m(\lri)}(s) = q^{8m-4} \frac{(q^2-q^{-s}) (q^2
      + q^{-s} + (q^4-1)q^{-2s} - q^{1-3s} + (q^4 - q^2 - q)
      q^{-4s})}{(1 - q^{1-2s})(1 - q^{2-3s})}. 
  \end{equation*}
\end{thm}

The four theorems are proved and discussed in
Sections~\ref{subsec:abscissa_estimates}, \ref{sec:expl_form_sl2},
\ref{sec:quaternion_case} and \ref{sec:sl3_Z3} respectively.  Whilst
all of the results in the current paper are of a local nature, they
might be best appreciated in the `global' context of representation
zeta functions of arithmetic groups, as explained to some extent below
and exposed, in much greater detail, in~\cite{AvKlOnVo10, AvKlOnVo10+,
  AvKlOnVo09}. The main technical tools of these papers are the
Kirillov orbit method, novel techniques from $\mfp$-adic integration
and Clifford theory. The current paper is closely related to these
works, complements them in parts and provides concrete examples of the
general methods developed in these papers. We tried, however, to keep
it reasonably self-contained, and hope that it might help the reader
appreciate these articles and their interconnections.

\smallskip

Let $\Gamma$ be an arithmetic subgroup of a connected, simply
connected semisimple algebraic group $\mathbf{G}$ defined over a
number field $k$, and assume that $\Gamma$ has the Congruence Subgroup
Property.  Relevant examples are groups of the form $\Gamma =
\SL_n(\gri)$, where $\gri$ is the ring of integers in a number
field~$k$, and $n\geq 3$.  The Congruence Subgroup Property and
Margulis super-rigidity imply that the `global' representation zeta
function $\zeta_\Gamma(s)$ of $\Gamma$ is an Euler product of `local'
representation zeta functions, indexed by places of~$k$; see
\cite[Proposition~1.3]{LaLu08}.  For example, if $\Gamma=\SL_n(\gri)$
with $n\geq 3$, we have
\begin{equation}\label{eqn:euler}
  \zeta_{\SL_n(\smallgri)}(s) = \zeta_{\SL_n(\C)}(s)^{|k:\Q|} \cdot
  \prod_{v} \zeta_{\SL_n(\smallgri_v)}(s)
\end{equation}
where each archimedean factor $\zeta_{\SL_n(\C)}(s)$ enumerates the
finite-dimensional, irreducible rational representations of the
algebraic group $\SL_n(\C)$ and, for every non-archimedean place $v$
of $k$, we denote by $\gri_v$ the completion of $\gri$ at~$v$, which
is a finite extension of the $p$\nobreakdash-adic integers $\Z_p$ if
$v$ prolongs $p$. In general, the local factors indexed by
non-archimedean places are representation zeta functions of FAb
compact $p$-adic analytic groups. In \cite{AvKlOnVo09} we provide
explicit, uniform formulae for the zeta functions of groups of the
form $\SL_3(\lri)$, in the case that $p > \max\{3,e\}$, where
$e=e(\lri,\Zp)$ denotes the absolute ramification index of~$\lri$.  In
the Euler product \eqref{eqn:euler} this condition is satisfied by all
but finitely many of the rings $\lri = \gri_v$. It is a natural,
interesting question to describe the local factors at the finitely
many `exceptional' places $v$ of~$k$, and in particular in the
non-generic case $p=3$.

In ~\cite{AvKlOnVo10+} we develop general methods to describe
representation zeta functions of certain `globally defined' FAb
compact $p$-adic analytic pro-$p$ groups. These include the principal
congruence subgroups of compact $p$-adic analytic groups featuring in
Euler products of representation zeta functions of `semisimple'
arithmetic groups, such as \eqref{eqn:euler}.  In particular, we
obtain there formulae for the zeta functions of principal congruence
subgroups of the form $\SL_3^m(\lri)$, provided the residue field
characteristic of $\lri$ is different from $3$ and $m \in \N$ is
permissible for the $\lri$-Lie lattice $\spl_3(\lri)$,
cf.~\cite[Theorem~E]{AvKlOnVo10+}. 

In Theorem~\ref{thm:SL3_p=3} of the current paper we complement
\cite[Theorem~E]{AvKlOnVo10+} (or, equivalently, the analogous result
in \cite{AvKlOnVo09}) by providing a formula for the representation
zeta functions of groups of the form $\SL_3^m(\lri)$, where $\lri$ is
an unramified extension of $\Z_3$ and $m\in\N$.  The formula differs
from the `generic' formula in~\cite[Theorem~E]{AvKlOnVo10+}, valid for
residue field characteristic $p\neq 3$, and is obtained by
computations akin to the algebraic approach developed
in~\cite{AvKlOnVo09}.  It is noteworthy that the `generic' formula
only depends on the residue field of the local ring $\lri$,
irrespective of ramification.  Whilst it is clear that this does not
hold in the case of residue field characteristic $3$, we do not know
how sensitive the zeta function is to ramification in this case.  It
is, for instance, an interesting open problem whether, for rings
$\lri_1$ and $\lri_2$ of residue field characteristic $3$ which have
the same inertia degree and ramification index over $\Z_3$, the
representation zeta functions of the groups $\SL_3(\lri_1)$ and
$\SL_3(\lri_2)$ are the same, or at least those of $\SL_3^m(\lri_1)$
and $\SL_3^m(\lri_2)$ for permissible $m \in \N$.

A phenomenon of the latter kind is exhibited in Theorem~\ref{thm:SL2},
in which we record formulae for zeta functions of groups of the form
$\SL_2^m(\lri)$, where $m\in\N$ is permissible for the Lie
lattice~$\spl_2(\lri)$.  The formula for $\zeta_{\SL_2^m(\lri)}(s)$ in
the generic case $p>2$ only depends on $m$ and on $q$, the residue
field cardinality of $\lri$.  The expression in the case $p=2$, on the
other hand, is sensitive to ramification, albeit only to the absolute
ramification index $e=e(\lri,\Z_2)$ of $\lri$.  In
Section~\ref{sec:full_sl_2} we employ Clifford theory to compute the
representation zeta function $\zeta_{\SL_2(\lri)}(s)$ in the case that
$p\geq e-2$. This reproduces, in the given case, a formula first
computed in~\cite[Theorem~7.5]{Ja06} for the zeta functions of groups
of the form $\SL_2(R)$, where $R$ is an arbitrary complete discrete
valuation ring with finite residue field of odd characteristic. We
record it here as it illustrates our broader and conceptually
different approach. We note that this formula, too, only depends on
the residue field cardinality and not, for instance, on
ramification. It is a challenge to establish analogous formulae in
residue field characteristic~$2$.  We record a formula for the
representation zeta functions of $\SL_2(\Z_2)$, based
on~\cite{NoWo76}, and a conjectural formula for its first principal
congruence subgroup.

Our methods also allow for explicit calculations for representation
zeta functions of norm one groups $\SL_1(\Ldr)$ of central division
algebras $\Ldr$ of Schur index $\ell$, say, over $\mfp$-adic
fields~$\lfi$, and their principal congruence subgroups.  Results
obtained by Larsen and Lubotzky in~\cite{LaLu08} suggest that such
`anisotropic' groups may actually be more tractable than their
`isotropic' counterparts.  In the case that $\ell$ is prime, we
compute the abscissa of convergence of $\zeta_{\SL_1(\Ldr)}(s)$ in
terms of Lie-theoretic data associated to the Lie algebra
$\spl_1(\Ldr)$; cf.\ Corollary~\ref{cor:abscissa_skew}.  This result,
which was first proved in~\cite{LaLu08} for general $\ell$, is an easy
application of some general estimates for the abscissae of convergence
of representation zeta functions of groups to which our methods are
applicable; cf. Theorem~\ref{thm:abscissa_estimate}.  Of special
interest is the case $\ell=2$, where $\Ldr$ is a non-split quaternion
algebra over $\lfi$ with maximal compact subring $\ldr$, say.  In this
situation we give, in Theorem~\ref{thm:quaternions}, formulae for the
representation zeta functions of the groups $\SL_1(\Ldr) =
\SL_1(\ldr)$ and $\SL_1^m(\ldr)$, $m \in \N$, which hold if $p$ is
large compared to the ramification index $e(\lfi,\Qp)$.

\smallskip 

\noindent \textit{Organisation.} The paper is organised as follows. In
Section~\ref{sec:p-adic_integral} we briefly recall the geometric
method, developed in~\cite{AvKlOnVo10+}, to describe representation
zeta functions of certain compact $p$-adic analytic pro-$p$ groups.
We show how it yields lower and upper bounds for abscissae of
convergence as in Theorem~\ref{thm:abscissa_estimate}.  In
Section~\ref{sec:sl2} we compute representation zeta functions
associated to groups of the form $\SL_2(\lri)$ and its principal
congruence subgroups, thereby proving Theorem~\ref{thm:SL2}.  Results
on representation zeta functions of subgroups of norm one groups of
central division algebras and, in particular,
Theorem~\ref{thm:quaternions} are obtained in
Section~\ref{sec:skew_fields}.  The computations for principal
subgroups of groups of the form $\SL_2(\lri)$ in residue field
characteristic $p=3$, resulting in Theorem~\ref{thm:SL3_p=3}, are
carried out in Section~\ref{sec:sl3_Z3}.

\smallskip

\noindent \textit{Notation.}  Our notation is the same as the one used
in~\cite{AvKlOnVo10+}.  Non-standard terms are briefly defined at
their first occurrence in the text.  Zeta function will always refer
to representation zeta function.  Throughout this paper, $\lri$
denotes a compact discrete valuation ring of characteristic $0$ and
residue field cardinality $q$, a power of a prime~$p$.  We write $F^*$
to denote the multiplicative group of a field $F$ and extend this
notation as follows.  For a non-trivial $\lri$-module $M$ we write
$M^* := M \setminus \mfp M$ and set $\{0\}^*=\{0\}$.


\section{Zeta functions as $\mfp$-adic integrals}
\label{sec:p-adic_integral}

Let $\lri$ be a compact discrete valuation ring of characteristic $0$,
with maximal ideal~$\mfp$. The residue field $\lri/\mfp$ is a finite
field of characteristic~$p$ and cardinality $q$, say.  Let $\lfi$ be
the field of fractions of~$\lri$.

\subsection{Integral formula} \label{subsec:basic_set_up} Let $\mfg$
be an $\lri$-Lie lattice such that $\lfi \otimes_{\lri} \mfg$ is
perfect.  In accordance with \cite[Section~2.1]{AvKlOnVo10+}, we call
$m \in \N_0$ permissible for $\mfg$ if the principal congruence Lie
sublattice $\mfg^m = \mfp^m \mfg$ is potent and saturable.  Almost all
non-negative integers $m$ are permissible for $\mfg$; see
\cite[Proposition~2.3]{AvKlOnVo10+}.  A key property of a potent and
saturable $\lri$-Lie lattice $\mfh$ is that the Kirillov orbit method
can be used to study the set $\Irr(H)$ of irreducible complex
characters of the $p$-adic analytic pro-$p$ group $H = \exp(\mfh)$,
which is associated to $\mfh$ via the Hausdorff series; see
\cite{Go09}.

Let $m \in \N_0$ be permissible for $\mfg$ and consider $\mathsf{G}^m
:= \exp(\mfg^m)$.  Then the orbit method provides a correspondence
between the elements of $\Irr(\mathsf{G}^m)$ and the co-adjoint orbits
of $\mathsf{G}^m$ on the Pontryagin dual $\Irr(\mfg^m) =
\Hom(\mfg^m,\C^*)$ of the compact abelian group~$\mfg^m$.  The radical
of $\omega \in \Irr(\mfg^m)$ is $\Rad(\omega) := \{ x \in \mfg^m \mid
\forall y \in \mfg^m : \omega([x,y]_\text{Lie}) = 1 \}$.  The degree
of the irreducible complex character represented by the co-adjoint
orbit of $\omega$ is equal to $\lvert \mfg^m : \Rad(\omega)
\rvert^{1/2}$, and the size of the co-adjoint orbit of $\omega$ is
equal to $\lvert \mfg^m : \Rad(\omega) \rvert$.  This shows that the
zeta function of $\mathsf{G}^m$ satisfies
\begin{equation} \label{Andrei's_formula} \zeta_{\mathsf{G}^m}(s) =
  \sum_{\omega \in \Irr(\mfg^m)} \lvert \mfg^m : \Rad(\omega)
  \rvert^{-(s+2)/2}.
\end{equation}
According to \cite[Lemma~2.4]{AvKlOnVo10+}, the Pontryagin dual of the
$\lri$-Lie lattice $\mfg^m$ admits a natural decomposition
$$
\Irr(\mfg^m) = \bigdotcup\nolimits_{n \in \N_0} \Irr_n(\mfg^m),
  \qquad \text{where } \Irr_n(\mfg^m) \cong
  \Hom_{\lri}(\mfg^m,\lri/\mfp^n)^*.
$$
Moreover, for each $n \in \N_0$ there is a natural projection of
$\lri$-modules $\Hom_\lri(\mfg^m,\lri) \rightarrow
\Hom_\lri(\mfg^m,\lri/\mfp^n)$, mapping $\Hom_\lri(\mfg^m,\lri)^*$
onto $\Hom_\lri(\mfg^m,\lri/\mfp^n)^*$.  We say that $\omega \in
\Irr_n(\mfg^m)$ has \emph{level} $n$ and that $w \in
\Hom_\lri(\mfg^m,\lri)^*$ is a representative of $\omega$ if $w$ maps
onto the appropriate element of $\Hom_\lri(\mfg^m,\lri/\mfp^n)^*$.

Let $\mathbf{b} := (b_1,\ldots,b_d)$ be an $\lri$-basis for the
$\lri$-Lie lattice $\mfg$, where $d = \dim_\lfi(\lfi \otimes_\lri
\mfg)$.  The structure constants $\lambda_{ij}^h$ of the $\lri$-Lie
lattice $\mfg$ with respect to $\mathbf{b}$ are encoded in the
commutator matrix
\begin{equation} \label{equ:commutator_matrix} \mathcal{R}(\mathbf{Y})
  := \mathcal{R}_{\mfg,\mathbf{b}}(\mathbf{Y}) = \left( \sum_{h=1}^d
    \lambda_{ij}^h Y_h \right)_{ij} \in \Mat_d(\lri[\mathbf{Y}] ),
\end{equation}
whose entries are linear forms in independent variables
$Y_1,\ldots,Y_d$.  We write $W(\lri) := \left( \lri^d \right)^* \cong
\Hom_\lri(\mfg^m,\lri)^*$ and set
\begin{equation*}
  \sigma(\mfg) := \min \left\{ \tfrac{1}{2} \rk_\lfi
    \mathcal{R}(\mathbf{y}) \mid \mathbf{y} \in
    W(\lri) \right\} \quad \text{and}
  \quad \rho(\mfg) := \max \left\{
    \tfrac{1}{2} \rk_{\lfi} \mathcal{R}(\mathbf{y}) \mid \mathbf{y} \in
    W(\lri) \right\}.
\end{equation*}
Note that this definition is independent of the choice of basis for
$\mfg$, and that both $\sigma(\mfg)$ and $\rho(\mfg)$ are integers,
because $\mcR(\mathbf{Y})$ is anti-symmetric.

The relevance of the commutator matrix in connection with
\eqref{Andrei's_formula} stems from \cite[Lemma~3.3]{AvKlOnVo10+},
which we record here as follows.

\begin{lem}\label{central_lem_neu}
  Let $m$ be permissible for $\mfg$, let $n \in \N$ and suppose that
  $\omega \in \Irr_n(\mfg^m)$ is represented by $w \in
  \Hom_\lri(\mfg^m,\lri)^*$.  Let $\pi$ denote a uniformiser
  for~$\lri$.  

  Then for every $z \in \mfg^m$ we have
  $$
  z \in \Rad(\omega) \quad \Longleftrightarrow \quad \underline{z}
  \cdot \mathcal{R}(\underline{w}) \equiv_{\mfp^{n-m}} 0,
  $$ where $\underline{z}$ and $\underline{w}$ denote the coordinate
  tuples of $z$ and $w$ with respect to the shifted $\lri$-basis
  $\pi^m\mathbf{b}$ for $\mfg^m$ and its dual
  $\pi^{-m}\mathbf{b}^\vee$ for $\Hom_\lri(\mfg^m,\lri)$. 
\end{lem}

By this lemma, the index $\lvert \mfg^m : \Rad(\omega) \rvert$ can be
expressed in terms of the elementary divisors of the matrix $
\mathcal{R}(\underline{w})$ which in turn one computes from its
minors.  The zeta function of the group $\mathsf{G}^m$, associated to
the principal congruence Lie sublattice $\mfg^m$, can thus be regarded
as a Poincar\'e series encoding the numbers of solutions of a certain
system of equations modulo $\mfp^n$ for all $n \in \N_0$.  Such
Poincar\'e series can be expressed as generalised Igusa zeta
functions, which are certain types of $\mfp$-adic integrals over the
compact space $\mfp \times W(\lri)$; cf.~\cite{AvKlOnVo10+} and
\cite{De91,Ig00}.

For $j \in \{1,\ldots,\rho(\mfg)\}$ and $\mathbf{y} \in W(\lri)$ we
define
\begin{align*}
  F_j(\mathbf{Y}) & = \{ f \mid f \text{ a $2j
    \times 2j$ minor of $\mathcal{R}(\mathbf{Y})$} \}, \\
  \lVert F(\mathbf{y}) \rVert_\mfp & = \max \{ \lvert f(\mathbf{y})
  \rvert_\mfp \mid f \in F \}.
\end{align*}
It is worth pointing out that the sets $F_j(\mathbf{Y})$ may be
replaced by sets of polynomials defining the same polynomial ideals.
Specifically, one could define $F_j(\mathbf{Y})$ to be the set of all
principal $2j \times 2j$ minors; see~\cite[Remark~3.16]{AvKlOnVo10+}.
It is the geometry of the varieties defined by the polynomials in
$F_j(\mathbf{Y})$ which largely determines the zeta function of
$\mathsf{G}^m$.  Of particular interest are `effective' resolutions of
their singularities; cf.\ \cite{AvKlOnVo10+}.  If $\lfi \otimes_\lri
\mfg$ is a semisimple Lie algebra, then the varieties defined by the
polynomials in $F_j(\mathbf{Y})$ admit a Lie-theoretic interpretation:
they yield a stratification of the Lie algebra defined in terms of
centraliser dimensions; cf.\ \cite[Section~5]{AvKlOnVo10+}.  We state
the integral formula derived in \cite[Section~3.2]{AvKlOnVo10+}.

\begin{pro}\label{pro:zeta=poincare}
  Let $\mfg$ be an $\lri$-Lie lattice such that $\lfi \otimes_{\lri}
  \mfg$ is a perfect $\lfi$-Lie algebra of dimension $d$.  Then for
  every $m \in \N_0$ which is permissible for $\mfg$ one has
  $$
  \zeta_{\mathsf{G}^m}(s) = q^{dm}
  \left(1+(1-q^{-1})^{-1} \mcZ_\lri(-s/2-1,\rho(s+2)-d-1) \right),
  $$
  with $\rho = \rho(\mfg)$ and
  \begin{equation} \label{equ:integral_neu} \mathcal{Z}_\lri(r,t) =
    \int_{(x,\mathbf{y}) \in \mfp \times W(\lri)} \lvert x
    \rvert_\mfp^t \prod_{1 \leq j \leq \rho(\mfg)} \frac{\lVert
      F_j(\mathbf{y}) \cup F_{j-1}(\mathbf{y}) x^2
      \rVert_\mfp^r}{\lVert F_{j-1}(\mathbf{y}) \rVert_\mfp^r} \,
    d\mu(x,\mathbf{y}),
  \end{equation}
  where $\mfp \times W(\lri) \subseteq \lri^{d+1}$ and the additive
  Haar measure $\mu$ is normalised to $\mu(\lri^{d+1}) = 1$.
\end{pro}

In studying the integral \eqref{equ:integral_neu}, it is useful to
distinguish between regular and irregular points of $W(\lri)$.  Let
$\mathcal{U}_1$ denote the subvariety of $\mathbb{A}^d$ defined by the
set of polynomials $F_\rho(\mathbf{Y})$ over $\lri$.  We write
$\mathbb{F}_q$ for the residue class field $\lri/\mfp$.  The reduction
of $\mathcal{U}_1$ modulo $\mfp$ is denoted by
$\overline{\mathcal{U}_1}$.  We call a point $\mathbf{a} \in
(\mathbb{F}_q^d)^*$, and any $\mathbf{y} \in W(\lri)$ mapping onto
$\mathbf{a}$, \emph{regular} if $\mathbf{a}$ is not an
$\mathbb{F}_q$-rational point of $\overline{\mathcal{U}_1}$.  A
functional $w \in \Hom_{\lri}(\mfg,\lri)^*$ and the representations
associated to the Kirillov orbits of the images of $w$ in
$\Hom_\lri(\mfg,\lri/\mfp^n)^*$, $n \in \N$, are said to be
\emph{regular}, if the co-ordinate vector $\mathbf{y} \in W(\lri)$
corresponding to $w$ is regular.  Points, functionals and
representations which are not regular are called \emph{irregular}.
 

\subsection{Adjoint versus co-adjoint action}\label{sec:adjointversuscoadjoint}
In the special case where $\mfg$ is an $\lri$-Lie lattice such that
$\lfi \otimes_\lri \mfg$ is semisimple, one can use the Killing form,
or a scaled version of it, to translate between co-adjoint orbits and
adjoint orbits.  This has some technical benefits when using the orbit
method, as illustrated in Sections~\ref{sec:skew_fields} and
\ref{sec:sl3_Z3}.

Let $\mfg$ be an $\lri$-lattice such that $\lfi \otimes_\lri \mfg$ is
semisimple, and suppose that $m$ is permissible for $\mfg$ so that the
$m$th principal congruence sublattice $\mfg^m = \mfp^m \mfg$ is
saturable and potent.  At the level of the Lie algebra $\lfi
\otimes_\lri \mfg$, the Killing form $\kappa$ is non-degenerate and
thus provides an isomorphism $\iota$ of $\lfi$-vector spaces between
$\lfi \otimes_\lri \mfg$ and its dual space $\Hom_\lfi(\lfi
\otimes_\lri \mfg,\lfi)$.  Moreover, this isomorphism is
$G$-equivariant for any $G \leq \Aut(\lfi \otimes_\lri \mfg)$.

At the level of the $\lri$-Lie lattices $\mfg$ and $\mfg^m$, the
situation is more intricate, because the restriction of $\kappa$, or a
scaled version $\kappa_0$ of it, may not be non-degenerate over
$\lri$.  Typically, the pre-image of $\Hom_\lri(\mfg,\lri)
\hookrightarrow \Hom_\lfi(\lfi \otimes_\lri \mfg, \lfi)$ under the
$\lfi$-isomorphism $\iota_0: \lfi \otimes_\lri \mfg \rightarrow
\Hom_\lfi(\lfi \otimes_\lri \mfg, \lfi)$ induced by $\kappa_0$ is an
$\lri$-sublattice of $\lfi \otimes_\lri \mfg$ containing $\mfg$ as a
sublattice of finite index.  For instance, if $\mfg$ is a simple Lie
algebra of Chevalley type, then it is natural to work with the
normalised Killing form $\kappa_0$ which is related to the ordinary
Killing form $\kappa$ by the equation $2 h^\vee \kappa_0 = \kappa$.
Here $h^\vee$ denotes the dual Coxeter number; e.g., the dual Coxeter
number for $\spl_n$ is $h^\vee = n$.

Irrespective of the detailed analysis required to translate carefully
between adjoint and co-adjoint orbits, we obtain from the general
discussion in \cite[Section~5]{AvKlOnVo10+} a useful description of
the parameters $\sigma(\mfg)$ and $\rho(\mfg)$, which were introduced
in Section~\ref{subsec:basic_set_up}.  Indeed, they can be computed in
terms of centraliser dimensions as follows:
\begin{equation}\label{equ:sigmarho_centraliser}
  \begin{split}
    \dim_\lfi(\lfi \otimes_\lri \mfg) - 2\sigma(\mfg) & = \max \{
    \dim_\lfi \Cen_{\lfi \otimes_\lri \mfg}(x) \mid x \in (\lfi
    \otimes_\lri \mfg)
    \setminus \{0\} \}, \\
    \dim_\lfi(\lfi \otimes_\lri \mfg) - 2\rho(\mfg) &  = \min \{
    \dim_\lfi \Cen_{\lfi \otimes_\lri \mfg}(x) \mid x \in (\lfi
    \otimes_\lri \mfg) \setminus \{0\} \}.
  \end{split}
\end{equation}


\subsection{General bounds for the abscissa of convergence}
\label{subsec:abscissa_estimates} 

In this section we derive general bounds for the abscissae of
convergence of zeta functions of compact $p$-adic analytic groups.  We
start by proving Theorem~\ref{thm:abscissa_estimate} which was stated
in the introduction.



\begin{proof}[Proof of Theorem~\ref{thm:abscissa_estimate}] Roughly
  speaking, the idea is that systematically overestimating the size of
  orbits in the co-adjoint action leads to a Dirichlet series
  $\psi_{\textup{low}}(s)$ which converges at least as well as
  $\zeta_{\mathsf{G}^m}(s)$ and hence provides a lower bound for
  $\alpha(\mathsf{G}^m)$.  Similarly, consistently underestimating the
  size of orbits leads to a Dirichlet series which converges no better
  than $\zeta_{\mathsf{G}^m}(s)$ and hence provides an upper bound for
  $\alpha(\mathsf{G}^m)$.  For this we use the description of
  $\zeta_{\mathsf{G}^m}(s)$ in \eqref{Andrei's_formula}.

  First we derive the lower bound for $\alpha(\mathsf{G}^m)$.
  Lemma~\ref{central_lem_neu}, in conjunction with the definition of
  $\rho$, implies that $\lvert \mfg^m:\Rad(\omega) \rvert \leq
  q^{2\rho n}$ for all $n \in \N_0$ and $\omega \in
  \Irr_n(\mfg^m)$.  Clearly, the Dirichlet series
  $$
  \psi_\textup{low}(s) := \sum_{n \in \N_0} \sum_{\omega \in
    \Irr_n(\mfg^m)} q^{-\rho n (s+2)}
  $$
  converges better than $\zeta_{\mathsf{G}^m}(s)$ and it suffices to
  show that the abscissa of convergence of $\psi_\textup{low}(s)$ is
  equal to $(d - 2\rho) \rho^{-1}$.  Indeed, this can easily be read
  off from the precise formula
  \begin{align*}
    \psi_\textup{low}(s) & = 1 + \sum_{n \in \N} (1-q^{-d}) q^{dn}
    q^{-\rho n (s+2)} \\
    & = 1 + (1-q^{-d}) q^{(d-2\rho)-\rho s} (1-q^{(d-2\rho)-\rho
      s})^{-1} \\
    & = (1 - q^{-2\rho - \rho s})(1 - q^{(d-2\rho)-\rho s})^{-1}.
  \end{align*}
  The argument for deriving the upper bound is essentially the same,
  but with a little extra twist.  Recall that $W(\lri) = \left( \lri^d
  \right)^*$.  Similarly as in \cite[Section~3.1]{AvKlOnVo10+} we
  consider a map $\nu : W(\lri) \rightarrow (\N_0 \cup
  \{\infty\})^{\lfloor d/2 \rfloor}$ which maps $\mathbf{y} \in
  W(\lri)$ to the tuple $\mathbf{a} = (a_1,\ldots,a_{\lfloor d/2
    \rfloor})$ such that
  \begin{itemize}
  \item[(i)] $a_1 \leq \ldots \leq a_{\lfloor d/2 \rfloor}$ and
  \item[(ii)] the elementary divisors of the anti-symmetric matrix
    $\mathcal{R}(\mathbf{y})$ are precisely $\mfp^{a_1}$, \ldots,
    $\mfp^{a_{\lfloor d/2 \rfloor}}$, each counted with multiplicity
    $2$, and one further divisor $\mfp^\infty$ if $d$ is odd.
  \end{itemize}
  The definition of $\sigma$ ensures that by forming the composition
  of $\nu$ with the projection $(a_1,\ldots,a_{\lfloor d/2 \rfloor})
  \mapsto (a_1,\ldots,a_\sigma)$ we obtain a map $\nu_{\textup{res}} :
  W(\lri) \rightarrow \N_0^\sigma$.  Clearly, $\nu_{\textup{res}}$ is
  continuous and hence locally constant.  Since $W(\lri)$ is compact,
  this implies that the image of $\nu_{\textup{res}}$ is finite.  From
  Lemma~\ref{central_lem_neu} we deduce that there is a constant $c
  \in \N_0$ such that for all $n \in \N_0$ and $\omega \in
  \Irr_n(\mfg^m)$ we have $\lvert \mfg^m:\Rad(\omega)
  \rvert \geq q^{2\sigma n -c}$.  Now a similar calculation as above
  gives the desired upper bound for $\alpha(\mathsf{G}^m)$.
\end{proof}

\begin{rem*}
  (1) According to \cite[Corollary~4.5]{LaLu08}, the abscissa of
  convergence $\alpha(G)$ is an invariant of the commensurability
  class of $G$.  Thus Theorem~\ref{thm:abscissa_estimate} provides a
  tool for bounding the abscissa of convergence of the zeta function
  of any FAb compact $p$-adic analytic group; see
  Corollary~\ref{cor:abscissa_semisimple}.

  (2) The algebraic argument given in the proof of
  Theorem~\ref{thm:abscissa_estimate} admits a geometric
  interpretation based on the integral formula in
  Proposition~\ref{pro:zeta=poincare}.  To obtain the lower bound one
  assumes that all points are `regular', to obtain the lower bound
  that all points are as `irregular' as possible.
\end{rem*}

For `semisimple' compact $p$-adic analytic groups, the lower bound in
Theorem~\ref{thm:abscissa_estimate} specialises to a result first
proved by Larsen and Lubotzky; see~\cite[Proposition~6.6]{LaLu08}.  We
formulate our more general result in this setting.  Recall that to any
compact $p$-adic analytic group $G$ one associates a $\Q_p$-Lie
algebra, namely $L(G) := \Q_p \otimes_{\Z_p} \mathfrak{h}$ where
$\mathfrak{h}$ is the $\Z_p$-Lie lattice associated to any saturable
open pro-$p$ subgroup $H$ of $G$.  This Lie algebra is an invariant of
the commensurability class of $G$.  Suppose that $L(G)$ is semisimple.
Then it decomposes as a sum $L(G) = S_1 \oplus \ldots \oplus S_r$ of
simple $\Q_p$-Lie algebras $S_i$.  For each $i \in \{1, \ldots, r\}$
the centroid $\lfi_i$ of $S_i$, viz.\ the ring of $S_i$-endomorphisms
of $S_i$ with respect to the adjoint action, is a finite extension
field of $\Q_p$, and $S_i$ is an absolutely simple $\lfi_i$-Lie
algebra.  The fields $\lfi_i$ embed into the completion $\C_p$ of an
algebraic closure of $\Q_p$.  The field $\C_p$ is, algebraically,
isomorphic to the field $\C$ of complex numbers.  Indeed, $\C_p$ and
$\C$ are algebraically closed and have the same uncountable
transcendence degree over~$\Q$.  Choosing an isomorphism between
$\C_p$ and $\C$, we define $\mathfrak{L} := \mathfrak{L}(G) := \C
\otimes_{\Q_p} L(G)$ and $\mathfrak{S}_i := \C \otimes_{\lfi_i} S_i$
for $i \in \{1, \ldots, r\}$.  We define $r_\textup{abs}(G)$ and
$\Phi_\textup{abs}(G)$ to be the absolute rank and the absolute root
system of $L(G)$; they are equal to the rank and the root system of
the semisimple complex Lie algebra~$\mathfrak{L}$.  We denote by
$h^\vee(\mathfrak{S})$ the dual Coxeter number of a simple complex Lie
algebra $\mathfrak{S}$.

\begin{cor} \label{cor:abscissa_semisimple} Let $G$ be a compact
  $p$-adic analytic group such that its associated $\Q_p$-Lie algebra
  $L(G)$ is semisimple and decomposes as described above.  Then the
  abscissa of convergence of $\zeta_G(s)$ satisfies
  $$
  \frac{2r_\textup{abs}(G)}{\lvert \Phi_\textup{abs}(G) \rvert} \leq
  \alpha(G) \leq \min_{i \in \{1,\ldots,r\}}
  \frac{\dim_{\C}(\mathfrak{S}_i)}{h^\vee(\mathfrak{S}_i)-1}-2.
  $$
\end{cor}

\begin{proof} By our remark, we may assume without loss of generality
  that $G = \exp(\mfg)$ is associated to a potent and saturable
  $\Z_p$-Lie lattice $\mfg$.  The lower bound for $\alpha(G)$ follows
  immediately from Theorem~\ref{thm:abscissa_estimate} and the equations
  \eqref{equ:sigmarho_centraliser} on noting that $r_\textup{abs}(G) =
  \dim(G) - 2\rho(\mfg)$ and $\lvert \Phi_\textup{abs}(G) \rvert =
  2\rho(\mfg)$.

  It remains to establish the upper bound.  Replacing $G$ by an open
  subgroup, if necessary, we may assume that $\mfg = \mathfrak{s}_1
  \oplus \ldots \oplus \mathfrak{s}_r$ and $G = G_1 \times \ldots
  \times G_r$, where for each $i \in \{1, \ldots, r\}$ the summand
  $\mathfrak{s}_i$ is a potent and saturable $\Z_p$-Lie lattice such
  that $\Q_p \otimes_{\Z_p} \mathfrak{s}_i$ is simple and $G_i =
  \exp(\mathfrak{s}_i)$.  As $\zeta_G(s) = \prod_{i=1}^r
  \zeta_{G_i}(s)$, we have $\alpha(G) = \min_{i \in \{1,\ldots,r\}}
  \alpha(G_i)$ and it is enough to bound $\alpha(G_i)$ for each $i \in
  \{1,\ldots,r\}$.

  Fix $i \in \{1,\ldots,r\}$ and write $\mathfrak{s} :=
  \mathfrak{s}_i$.  As before, the centroid $\lfi$ of $\Q_p
  \otimes_{\Z_p} \mathfrak{s}$ is a finite extension of $\Q_p$, and
  $\Q_p \otimes_{\Z_p} \mathfrak{s}$ is an absolutely simple
  $\lfi$-Lie algebra.  Without loss of generality we may regard
  $\mathfrak{s}$ as an $\lri$-Lie lattice, where $\lri$ is the ring of
  integers of $\lfi$.  Writing $\mathfrak{S} = \C \otimes_\lri
  \mathfrak{s}$, we deduce from Theorem~\ref{thm:abscissa_estimate}
  that it suffices to show: $\sigma(\mathfrak{s}) \geq
  h^\vee(\mathfrak{S})-1$.  It is clear that $2\sigma(\mathfrak{s})$
  is greater or equal to the dimension of a non-zero co-adjoint orbit
  of $\mathfrak{S}$ of minimal dimension.  According to \cite[Section
  5.8]{BoKr79}, every sheet of $\mathfrak{S}$ contains a unique
  nilpotent orbit, and the dimension of a minimal nilpotent orbit in
  $\mathfrak{S}$ is equal to $2 h^\vee(\mathfrak{S})-2$;
  see~\cite[Theorem 1]{Wa99}. It follows that $\sigma(\mathfrak{s})
  \geq h^\vee(\mathfrak{S})-1$.
\end{proof}

It is worth pointing out that the absolute rank and the size of the
absolute root system of a semisimple Lie algebra grow proportionally
at the same rate under restriction of scalars; hence, if $G$ is
defined over an extension $\lri$ of $\Z_p$, then it is natural to work
directly with the invariants of the Lie algebra over $\lfi$, without
descending to~$\Q_p$.  A similar remark applies to the upper bound in
Corollary~\ref{cor:abscissa_semisimple}.  For instance, for the family
of special linear groups $\SL_n(\lri)$, $n \in \N$, we obtain the
estimates 
$$
2/n \leq \alpha(\SL_n(\lri)) \leq n-1,
$$
reflecting the fact that $\spl_n(\C)$ has rank $n-1$, a root system of
size $n^2-n$, dimension $n^2-1$ and dual Coxeter number
$h^\vee(\spl_n(\C)) = n$.  More generally, we note that
Corollary~\ref{cor:abscissa_semisimple} provides upper bounds for the
abscissae of convergence of zeta functions of groups corresponding to
classical Lie algebras which are linear in the rank.  We further
remark that for `isotropic simple' compact $p$-adic analytic groups
the abscissa of convergence is actually bounded from below by $1/15$;
see~\cite[Theorem~8.1]{LaLu08}.

Another consequence of Theorem~\ref{thm:abscissa_estimate} is
recorded as Corollary~\ref{cor:abscissa_skew} in
Section~\ref{sec:skew_sub_1}.


\section{Explicit formulae for $\SL_2(\lri)$ and its principal
  congruence subgroups}\label{sec:sl2}

In this section we use the setup from
Sections~\ref{subsec:basic_set_up}, \ref{sec:adjointversuscoadjoint}
and \cite[Section~5]{AvKlOnVo10+} to compute explicitly the zeta
functions of `permissible' principal congruence subgroups of the
compact $p$-adic analytic group $\SL_2(\lri)$, where $\lri$ denotes a
compact discrete valuation ring of characteristic $0$.  As before, we
write $\mfp$ for the maximal ideal of $\lri$; the characteristic and
cardinality of the residue field $\lri/\mfp$ are denoted by $p$ and
$q$.  We write $e(\lri,\Z_p)$ for the absolute ramification index of
$\lri$.

\subsection{Principal congruence subgroups of
  $\SL_2(\lri)$} \label{sec:expl_form_sl2}

Our aim in this section is to prove Theorem~\ref{thm:SL2} which was
stated in the introduction.  It provides explicit formulae for the
zeta functions of principal congruence subgroups $\SL_2^m(\lri)$ for
permissible $m$, with no restrictions if $p > 2$ and for unramified
$\lri$ if $p=2$.


\begin{rem*}
  In fact, our proof also supplies an explicit formula, if $p=2$ and
  $e(\lri,\Z_2) >1$, but this formula is not as concise as the ones
  stated in Theorem~\ref{thm:SL2}.  It is noteworthy that in this
  special case the formula only depends on the ramification index
  $e(\lri,\Z_2)$, but not on the more specific isomorphism type of the
  ring $\lri$.  It would be interesting to investigate what happens
  for `semisimple' groups of higher dimensions; already for
  $\SL_3(\lri)$ the matter remains to be resolved; cf.\
  Section~\ref{sec:sl3_Z3}.
\end{rem*}

\begin{proof}[Proof of Theorem~\ref{thm:SL2}]
  Let $m \in \N$ be permissible for $\spl_2(\lri)$.  We need to
  compute the integral~\eqref{equ:integral_neu} over $\mfp \times
  W(\lri)$, where $W(\lri) = \big( \lri^3 \big)^*$.  It is easy to
  write down the commutator matrix $\mathcal{R}(\mathbf{Y})$ for the
  $\lri$-Lie lattice $\spl_2(\lri)$, and one verifies immediately that
  $\rho = 1$.  Indeed, working with the standard $\lri$-basis
  $\mathbf{e} = \bigl(
  \begin{smallmatrix} 0 & 1 \\ 0 & 0 \end{smallmatrix} \bigr)$,
  $\mathbf{f} = \bigl(
  \begin{smallmatrix} 0 & 0 \\ 1 & 0 \end{smallmatrix} \bigr)$,
  $\mathbf{h} = \bigl(
  \begin{smallmatrix} 1 & 0 \\ 0 & -1 \end{smallmatrix} \bigr)$ of
  $\spl_2(\lri)$ one obtains
  \begin{equation}\label{comm_matr_sl2}
    \mathcal{R}(\mathbf{Y}) =
    \begin{pmatrix}
      0    & Y_3   & -2Y_1 \\
      -Y_3 & 0     & 2Y_2  \\
      2Y_1 & -2Y_2 & 0
    \end{pmatrix}.
  \end{equation}
  In view of \eqref{comm_matr_sl2} we distinguish two cases.

  First suppose that $p > 2$.  In this case it is easily seen that
  $$
  \max \{ \lvert f(\mathbf{y}) \rvert_\mfp \mid f \in F_1(\mathbf{Y})
  \} \cup \{ \lvert x^2 \rvert_\mfp \} = 1 \quad \text{for all $x \in
    \mfp$ and $\mathbf{y} \in W(\lri)$.}
  $$
  Thus the integral~\eqref{equ:integral_neu} takes the form
  $$
  \mathcal{Z}_\lri(r,t) = \int_{(x,\mathbf{y}) \in \mfp \times W(\lri)}
  \lvert x \rvert_\mfp^t \, d\mu(x,\mathbf{y}).
  $$
  As
  $$
  \int_{x \in \mfp} \lvert x \rvert_\mfp^s \, d\mu(x) = \frac{(1 -
    q^{-1}) q^{-1-s}}{1 - q^{-1-s}}
  $$
  and $\mu(W(\lri)) = 1 - q^{-3}$ we obtain
  $$
  \mathcal{Z}_\lri(r,t) = \frac{(1 - q^{-1}) q^{-1-t} (1 - q^{-3})}{1
    - q^{-1-t}}
  $$ so that, by Proposition~\ref{pro:zeta=poincare}, we have
  \begin{equation*}
    \zeta_{\SL_2^m(\lri)}(s) = q^{3m} \left( 1 + (1-q^{-1})^{-1}
      \mathcal{Z}_\lri (-s/2-1,s-2) \right) = \frac{q^{3m}(1-q^{-2-s})}{1-q^{1-s}}.
  \end{equation*}

  Now consider the exceptional case $p=2$.  Put $e := e(\lri,\Z_2)$.  Defining  
  $$
  W(\lri)^{[j]} :=
  \begin{cases}
    \{ \mathbf{y} \in W(\lri) \mid y_3 \in \mfp^j \setminus \mfp^{j+1}
    \}
    & \text{if $0 \leq j \leq e-1$,} \\
    \{ \mathbf{y} \in W(\lri) \mid y_3 \in 2 \lri \} & \text{if
      $j=e$,}
  \end{cases}
  $$
  we write $W(\lri)$ as a disjoint union $W(\lri) = W(\lri)^{[0]} \,
  \dotcup \, W(\lri)^{[1]} \, \dotcup \, \ldots \, \dotcup \,
  W(\lri)^{[e]}$.  From \eqref{comm_matr_sl2} we see that, for all $x
  \in \mfp^i \setminus \mfp^{i+1}$ and $\mathbf{y} \in W(\lri)^{[j]}$,
  where $i \in \N$ and $1 \leq j \leq e$,
  \begin{equation*}
    \max \{ \lvert f(\mathbf{y}) \rvert_\mfp \mid f \in F_1(\mathbf{Y})
    \} \cup \{ \lvert x^2 \rvert_\mfp \} = q^{- 2 \min \{i,j,e\}}.
  \end{equation*}
  Furthermore, we note that $\mu(W(\lri)^{[0]}) = (1-q^{-1})$,
  $\mu(W(\lri)^{[j]}) = (1-q^{-1})(1-q^{-2})q^{-j}$ for $1 \leq j \leq
  e-1$, and $\mu(W(\lri)^{[e]}) = (1-q^{-2})q^{-e}$.  Thus the
  integral~\eqref{equ:integral_neu} takes the form
  \begin{align*}
    \mathcal{Z}_\lri(r,t) & = \mu(W(\lri)^{[0]}) \int_{x \in \mfp}
    \lvert
    x \rvert_\mfp^t \, d\mu(x) \\
    & \quad + \sum_{j=1}^e \mu(W(\lri)^{[j]}) \left( \sum_{i=1}^{j-1}
      q^{-2ir} \int_{x \in \mfp^i \setminus \mfp^{i+1}} \lvert x
      \rvert_\mfp^t \, d\mu(x) + q^{-2jr} \int_{x \in \mfp^j} \lvert x
      \rvert_\mfp^t \, d\mu(x) \right)
  \end{align*}
  which in turn yields an explicit formula for
  $\zeta_{\SL_2^m(\lri)}(s)$, again by
  Proposition~\ref{pro:zeta=poincare}.

  In the special case $e = e(\lri,\Z_2) = 1$ the resulting formula is
  as concise as for $p>2$: indeed, we have
  $$
  \mathcal{Z}_\lri(r,t) = \left( (1-q^{-1}) + (1-q^{-2}) q^{-1-2r}
  \right) \frac{(1-q^{-1}) q^{-1-t}}{1-q^{-1-t}}
  $$
  and consequently
  $$
  \zeta_{\SL_2^m(\lri)}(s) = q^{3m} \left( 1 + (1-q^{-1})^{-1}
    \mathcal{Z}_\lri(-s/2-1,s-2) \right) = \frac{q^{3m}(q^2 -
    q^{-s})}{1 - q^{1-s}}.
  $$
\end{proof}

Computer aided calculations suggest the following
conjecture.
\begin{con}\label{con:SL_2^1}
  The  zeta function of $\SL_2^1(\Z_2)$ is given by
  $$
  \zeta_{\SL_2^1(\Z_2)}(s) = \frac{2^{5}(2^2 - 2^{-s})}{1 -
    2^{1-s}}.
  $$
\end{con}


\subsection{Clifford theory}\label{sec:Clifford_basics} We briefly
recall some applications of basic Clifford theory to representation
zeta functions.  For more details we refer to
\cite[Section~7]{AvKlOnVo10+}.  Let $G$ be a group, and $N
\trianglelefteq G$ with $\lvert G : N \rvert < \infty$.  For $\theta
\in \Irr(N)$, let $I_G(\theta)$ denote the inertia group of $\theta$
in $G$, and $\Irr(G , \theta)$ the set of all irreducible characters
$\rho$ of $G$ such that $\theta$ occurs as an irreducible constituent
of the restricted character $\res_N^G(\rho)$.  One shows that, if $N$
admits only finitely many irreducible characters of any given degree,
then so does $G$ and
\begin{equation}\label{equ:hochheben}
  \zeta_G(s) = \sum_{\theta
    \in \Irr(N)} \theta(1)^{-s} \cdot \lvert G : I_G(\theta) \rvert^{-1-s}
  \zeta_{G,  \theta}(s),
\end{equation}
where
$$
\zeta_{G, \theta}(s) := \theta(1)^s \lvert G : I_G(\theta)
\rvert^s \sum_{\rho \in \Irr(G , \theta)} \rho(1)^{-s}.
$$
In the special case where $\theta$ extends to an irreducible character
$\hat \theta$ of $I_G(\theta)$, there is an effective description of
the elements $\rho \in \Irr(G, \theta)$, and one has $\zeta_{G,
  \theta}(s) = \zeta_{I_G(\theta)/N}(s)$.  There are several basic
sufficient criteria for the extendability of $\theta$; cf.\
\cite[Chapter~19]{Hu98}.


\subsection{The group $\SL_2(\lri)$} \label{sec:full_sl_2} In this
section we combine the Kirillov orbit method and basic Clifford theory
to compute explicitly the zeta function of the compact $p$-adic
analytic group $\SL_2(\lri)$.  The zeta function of the group
$\SL_2(R)$, where $R$ is an arbitrary compact discrete valuation ring
of odd residue characteristic, was first computed by Jaikin-Zapirain
by means of a different approach; see
\cite[Theorem~7.5]{Ja06}. 

For our approach we assume that $p-2 \geq e$ where $e :=
e(\lri,\Z_p)$; in particular, this implies $p>2$.  Then the $\lri$-Lie
lattice $\spl_2^1(\lri)$ and the corresponding pro-$p$ group
$\SL_2^1(\lri)$ are potent and saturable; see
\cite[Proposition~2.3]{AvKlOnVo10+}.  This means that the orbit method
can be applied to describe the irreducible characters of
$\SL_2^1(\lri)$.

Write $G := \SL_2(\lri)$ and $N := \SL_2^1(\lri)$.  Clifford theory,
as indicated in Section~\ref{sec:Clifford_basics}, provides a
framework to link $\Irr(G)$ and $\Irr(N)$.  Put $\mfg := \spl_2(\lri)$
and $\mfn := \spl_2^1(\lri) = \mfp \mfg$.  The Kirillov orbit method
links characters $\theta \in \Irr(N)$ to co-adjoint orbits of $N$ on
$\Hom_\lri(\mfn,\lri)$.  Choose a uniformiser $\pi$ of $\lri$.  Via
the $G$-equivariant isomorphism of $\lri$-modules $\mfg \rightarrow
\mfn$, $x \mapsto \pi x$, we can link co-adjoint orbits on
$\Hom_\lri(\mfn,\lri)$ to co-adjoint orbits on $\Hom_\lri(\mfg,\lri)$.

We follow closely the approach outlined in
\cite[Section~5]{AvKlOnVo10+}, which uses the normalised Killing form
to translate between the adjoint action of $G$ on $\mfg$ and the
co-adjoint action of $G$ on $\Hom_\lri(\mfg,\lri)$.  The dual Coxeter
number of $\spl_2$ is $h^{\vee} = 2$ so that the normalised Killing
form
$$
\kappa_0: \mfg \times \mfg \rightarrow \lri, \quad \kappa_0(x,y) = (2
h^{\vee})^{-1} \Tr(\ad(x) \ad(y))
$$
has the structure matrix
$$
[ \kappa_0( \cdot, \cdot)]_{(\mathbf{h},\mathbf{e},\mathbf{f})} = 
\left(
\begin{smallmatrix}
  2 & 0 & 0 \\
  0 & 0 & 1 \\
  0 & 1 & 0
\end{smallmatrix}
\right)
$$
with respect to the basis
\begin{equation*}
  \mathbf{h} = \left(
  \begin{smallmatrix}
    1 & 0 \\
    0 & -1
  \end{smallmatrix} \right), \quad \mathbf{e} = \left(
  \begin{smallmatrix}
    0 & 1 \\
    0 & 0
  \end{smallmatrix} \right), \quad \mathbf{f} = \left(
  \begin{smallmatrix}
    0 & 0 \\
    1 & 0
  \end{smallmatrix} \right).
\end{equation*}
As $p > 2$, the form $\kappa_0$ is `non-degenerate' over $\lri$ and
induces a $G$-equivariant isomorphism of $\lri$-modules $\mfg
\rightarrow \Hom_\lri(\mfg,\lri)$, $x \mapsto \kappa_0(x,\cdot)$.  We
obtain a $G$-equivariant commutative diagram
\begin{equation}\label{equ:diagram}
\begin{CD}
  \mfg^* @>{\cong}>>  \Hom_{\lri}(\mfg,\lri)^* \\
  @VVV @VVV \\
  (\mfg / \mfp^n \mfg)^* @>{\cong}>> \Hom_\lri(\mfg/\mfp^n\mfg,\lri /
  \mfp^n)^* @>{\cong}>> \Irr_n(\mfn) \\
  @VVV @VVV \\
  \spl_2(\F_q)^* @>{\cong}>> \Hom_{\F_q}(\spl_2(\F_q), \F_q)^*
\end{CD}
\end{equation}
where the last row is obtained by reduction modulo $\mfp$ and we have
used the isomorphism $\lri/\mfp \cong \F_q$.  Following the approach
taken in \cite[Section~7]{AvKlOnVo10+}, we are interested in the
orbits and centralisers of elements $x \in \mfg$ and their reductions
$\overline{x}$ modulo $\mfp$ under the adjoint action of $G$.

In order to apply Clifford theory, we require an overview of the
elements in $\spl_2(\mathbb{F}_q)$ up to conjugacy under the group
$\GL_2(\F_q)$.  We distinguish four different types, labelled $0$,
$1$, $2$a,~$2$b.  The total number of elements of each type and the
isomorphism types of their centralisers in $\SL_2(\F_q)$ are
summarised in Tables~\ref{table_sl2_1} and \ref{table_sl2_2}; see
Appendix~\ref{sec:aux_sl2} for a short discussion.  We remark that in
this particular case all elements are regular.

\begin{table}
  \centering
  \begin{tabular}{|c|l|l|l|l|}
    \hline
    type &  & no.\ of orbits & size of each orbit & total
    number \\
    \hline\hline
    0 & -- & $1$ & $1$ & $1$ \\
    1 & regular & $1$ & $q^2-1$ & $q^2 -1$ \\
    2a & regular & $(q-1)/2$ & $q^2 + q$ & $(q^2-1) q/2$ \\
    2b & regular & $(q-1)/2$ & $q^2 - q$ & $(q-1)^2 q/2$ \\
    \hline
  \end{tabular}
  \smallskip
  \caption{Orbits in $\spl_2(\mathbb{F}_q)$ under
    conjugacy by $\GL_2(\F_q)$ where $q=p^r$}
  \label{table_sl2_1}
\end{table}
\begin{table}
  \centering
  \begin{tabular}{|c|l|l|}
    \hline
    type & & centraliser in $\SL_2(\F_q)$ \\
    \hline\hline
    0 & -- & $\SL_2(\F_q)$ \\
    1 & regular\ & $\mu_2(\F_q) \times \F_q^+ \cong C_2 \times C_p^r$ \\
    2a & regular\ & $\F_q^* \cong C_{q-1}$ \\
    2b & regular\ & $\ker(N_{\F_{q^2} \vert \F_q}) \cong C_{q+1}$ \\
    \hline
  \end{tabular}
  \smallskip
  \caption{Centralisers in $\SL_2(\F_q)$ of elements of
    $\spl_2(\mathbb{F}_q)$ where $q = p^r$}
  \label{table_sl2_2}
\end{table}

Corollary~7.6 in \cite{AvKlOnVo10+} provides 

\begin{lem} \label{lem:lift} Let $x \in \spl_2(\lri)^*$, and
  let $\overline{x} \in \spl_2(\F_q)^*$ denote the reduction
  of $x$ modulo $\mfp$.  Then
  $$
  \Cen_{\SL_2(\lri)}(\overline{x}) =
  \Cen_{\SL_2(\lri)}(x) \SL_2^1(\lri).
  $$
\end{lem}

\begin{rem*}
  Alternatively, a direct argument shows that under the conjugation
  action by $\GL_2(\lri)$ the elements of $\spl_2(\lri)^*$ fall into
  orbits represented by matrices of the form
  \begin{align*}
    \left( \begin{smallmatrix}
        0 & 1 \\
        \pi^n \nu & 0
      \end{smallmatrix} \right) & \text{ with $n \in \N \cup
      \{\infty\}$
      and $\nu \in \lri \setminus \mfp$,} \\
    \left( \begin{smallmatrix}
        \lambda & 0 \\
        0 & -\lambda
      \end{smallmatrix} \right) & \text{ with $\lambda \in \lri
      \setminus \mfp$,} \\
    \left( \begin{smallmatrix}
        0 & 1 \\
        \mu & 0
      \end{smallmatrix} \right) & \text{ with $\mu \in \lri \setminus
      \mfp$, not a square modulo $\mfp$.}
  \end{align*}
  A short computation shows that the centralisers of these matrices in
  $\SL_2(\lri)$ are, respectively,
  \begin{align*}
    \left\{ \left( \begin{smallmatrix}
          a & b \\
          \pi^n \nu b & a
        \end{smallmatrix} \right) \right. & \left. \mid a,b \in \lri
      \text{ with } a^2 - \pi^n \nu b^2 = 1 \right\}, \\
    \left\{ \left( \begin{smallmatrix}
          a & 0 \\
          0 & b
        \end{smallmatrix} \right) \right. & \left. \mid a,b \in \lri
      \text{ with } ab = 1 \right\}, \\
    \left\{ \left( \begin{smallmatrix}
          a & b \\
          \mu b & a
        \end{smallmatrix} \right) \right. & \left. \mid a,b \in \lri
      \text{ with } a^2 - \mu b^2 = 1 \right\}.
  \end{align*}

  We may assume that $x$ is one of the listed
  representatives.  The centraliser of $\overline{x}$ in
  $\SL_2(\F_q)$ has a similar form as that of $x$; cf.\ our
  discussion above.  Aided by Hensel's lemma, one successfully lifts
  any given element $g_0 \in
  \Cen_{\SL_2(\lri)}(\overline{x})$ to an element $g \in
  \Cen_{\SL_2(\lri)}(\overline{x})$.
\end{rem*}

In any case, if $\theta \in \Irr(N)$ is represented by the adjoint
orbit of $x \vert_n := x + \mfp^n \mfg \in (\mfg/ \mfp^n \mfg)^*$, we
deduce that $\Cen_G(x) N = \Cen_G(x \vert_n) N =
\Cen_G(\overline{x})$, hence
$$
I_G(\theta) = \Cen_G(\overline{x}) \qquad \text{and} \qquad
I_G(\theta)/N \cong \Cen_{\SL_2(\F_q)}(\overline{x}).
$$
The isomorphism types of these groups are given by
Table~\ref{table_sl2_2}, where $\mu_2(\mathbb{F}_q)$ is the group of
square roots of unity in $\mathbb{F}_q^*$, we denote by
$\mathbb{F}_q^+$ the additive group of the field $\mathbb{F}_q$, and
$\ker(N_{\F_{q^2} \vert \F_q})$ is the multiplicative group of
norm-$1$ elements in $\mathbb{F}_{q^2} \vert \mathbb{F}_q$.

Note that, if $\theta$ is of type $2$a or $2$b, the inertia group
quotient $I_G(\theta)/N$ has order coprime to $p$.  This implies that
$\theta$ can be extended to a character $\hat \theta$ of
$I_G(\theta)$.  If $\theta$ is of type $1$, the inertia group is a
Sylow pro-$p$ subgroup of $\SL_2(\lri)$, and we can draw the same
conclusion based on \cite[Theorem~2.3]{AnNi08}: being an algebra
group, the character degrees of $I_G(\theta)$ are powers of $q$, and
since $\lvert I_G(\theta):N \rvert = q$, Lemma~7.4 in
\cite{AvKlOnVo10+} shows that $\theta$ extends to a character $\hat
\theta$ of $I_G(\theta)$. (If $2e<p-1$, the Sylow pro-$p$ subgroup of
$\SL_2(\lri)$ corresponds to a potent and saturable $\lri$-Lie
lattice; cf.\ \cite{Kl05}.  Hence one can apply
\cite[Corollary~3.2]{AvKlOnVo10+} to show that its character degrees
are powers of $q$.)  Since the inertia group quotients are abelian, it
follows that in all cases $\zeta_{G,\theta}(s) =
\zeta_{I_G(\theta)/N}(s) = \lvert I_G(\theta) : N \rvert$.

The final task is to connect the formula
\begin{equation} \label{equ:Clifford} \zeta_G(s) = \sum_{\theta \in
    \Irr(N)} \theta(1)^{-s} \cdot \lvert G : I_G(\theta) \rvert^{-1-s}
  \lvert I_G(\theta):N \rvert,
\end{equation}
resulting from \eqref{equ:hochheben}, and the explicit formula for
$\zeta_N(s) = \sum_{\theta \in \Irr(N)} \theta(1)^{-s}$ which we
obtained in Section~\ref{sec:expl_form_sl2}.  We parameterise the
non-trivial characters in $\Irr(N)$ by means of the affine cone
$W(\lri) = \big( \lri^3 \big)^*$, uniformly for all levels, as in the
integral formula~\eqref{equ:integral_neu}.  The affine cone $W(\lri)$
decomposes as a disjoint union of subsets $W(\lri)^{\text{[$1$]}}$,
$W(\lri)^{\text{[$2$a]}}$ and $W(\lri)^{\text{[$2$b]}}$ corresponding
to representations of types $1$, $2a$ and $2b$.  According to our
explicit element count modulo $\mfp$ (see Table~\ref{table_sl2_1}), we
have $\mu(W(\lri)^{\text{[$1$]}}) = q^{-1} (1-q^{-2})$,
$\mu(W(\lri)^{\text{[$2$a]}}) = (1-q^{-2})/2$ and
$\mu(W(\lri)^{\text{[$2$b]}}) = (1-q^{-1})^2 /2$.  With this
preparation we may write
$$
\zeta_N(s) = 1 + q^3 \left( \mu(W(\lri)^{\text{[$1$]}}) +
  \mu(W(\lri)^{\text{[$2$a]}}) + \mu(W(\lri)^{\text{[$2$b]}})
\right) (1 - q^{1-s})^{-1},
$$
and this yields
\begin{equation}\label{equ:SL2_Andrei}
\begin{split}
  \zeta_{\SL_2(\lri)}(s) & = \zeta_{\SL_2(\F_q)}(s) + q^3
  \left( \mu(W(\lri)^{\text{[$1$]}}) ((q^2 -1)/2)^{-1-s} 2q  \right. + \\
  & \qquad \left. \mu(W(\lri)^{\text{[$2$a]}}) (q^2 + q)^{-1-s} (q-1) \right. + \\
  & \qquad \left. \mu(W(\lri)^{\text{[$2$b]}}) (q^2 -q)^{-1-s} (q+1)
  \right) (1 - q^{1-s})^{-1} \\
  & = 1 + X_1 + \tfrac{q-3}{2} X_2 + 2 X_3 + \tfrac{q-1}{2} X_4 +
  2 X_5  + \phantom{x} \\
  & \qquad (1-qX_1)^{-1} \left( 4q X_2 X_5 + \tfrac{q^2 -1}{2} X_1 X_4
    + \tfrac{(q-1)^2}{2} X_1 X_2 \right)
\end{split}
\end{equation}
where $X_1 = q^{-s}$, $X_2 = (q+1)^{-s}$, $X_3 = ((q+1)/2)^{-s}$, $X_4
= (q-1)^{-s}$ and $X_5 = ((q-1)/2)^{-s}$.  The last formula is in
agreement with \cite[Theorem~7.5]{Ja06}.

It is worth pointing out that $\zeta_{\SL_2(\lri)}(s)$, unlike
$\zeta_{\SL_2^1(\lri)}(s)$, cannot be written as a rational function
in $q^{-s}$.  In \cite[Theorem~A]{AvKlOnVo10+}, we establish in a
rather general context local functional equations for the zeta
functions associated to families of pro-$p$ groups, such as
$\SL_2^1(\lri)$; cf.~Theorem~\ref{thm:SL2}.  It would be very
interesting if these could be meaningfully extended to the zeta
functions of larger compact $p$-adic analytic groups, such as
$\SL_2(\lri)$.

\begin{rem*}
  It would be interesting to see if the use of Clifford theory, as
  explored in the current section, can also be employed to establish
  the conjectural formula for the zeta function of $\SL_2^1(\Z_2)$
  stated in Conjecture~\ref{con:SL_2^1}.  For this one would start
  from our analysis of the zeta function $\zeta_{\SL_2^2(\Z_2)}(s)$;
  cf.~Theorem~\ref{thm:SL2}.  
\end{rem*}

\subsection{The group $\SL_2(\Z_2)$} \label{sec:full_sl_2_p=2} The
paper~\cite{NoWo76} contains an explicit construction of the
irreducible representations of the group $\SL_2(\Z_2)$.  This is
achieved by decomposing Weil representations associated to binary
quadratic forms and by considering tensor products of certain
components of such Weil representations.  To complement
\eqref{equ:SL2_Andrei} we record the following immediate consequence
of the work in~\cite{NoWo76}.

\begin{thm} 
  We have
  $$
  \zeta_{\SL_2(\Z_2)}(s) = \frac{(4 - 2^{1-s} - 5 \cdot 2^{1-2s} +
    2^{1-3s}) + 3^{-s} (28 + 2^{1-s} - 5 \cdot 2^{1-2s} + 2^{1-3s})}{1
    - 2^{1-s}}.
  $$
\end{thm}

\begin{proof}
  It follows by inspection of the classification results in \cite[pp.\
  522--524]{NoWo76} that the continuous irreducible characters of
  $G:=\SL_2(\Z_2)$ all have degrees of the form $2^i$ or $3\cdot2^i$,
  for $i\in\N_0$. Concerning $2$-power-degree characters, one has
  $$
  r_1(G) = 4,\quad r_2(G) = 6, \quad r_{2^2}(G)=2, \quad\text{ and
  }\quad r_{2^i}(G) = 3\cdot 2^{i-2}\text{ for }i\geq 3.
  $$ 
  This yields
  \begin{equation}
    \sum_{i=0}^\infty r_{2^i}(G)(2^{-s})^i=\frac{4 -2\cdot(2^{-s}) - 10
      \cdot 2^{-2s} + 2 \cdot 2^{-3s}}{1-2\cdot 2^{-s}}.\label{equ:2^i}
  \end{equation}
  The numbers of characters of degree $3\cdot 2^i$, for $i\in\N_0$,
  are given by
  $$
  r_3(G) = 28, \quad r_{3\cdot 2}(G) = 58, \quad r_{3\cdot 2^2}(G) =
  106, \quad\text{ and }\quad r_{3\cdot 2^i}(G) = 107\cdot 2^{i-2}
  \text{ for }i\geq 3.
  $$ Indeed, for $i\geq3$ the characters of degree
  $3\cdot 2^i$ come from levels $i+1,i+2,i+3$ and $i+4$, with
  contributions from these levels of $2^{i-2}$, $2^{i-1}$,
  $5\cdot2^{i+1}$ and $2^{i+4}$ characters, respectively. One checks
  that $2^{i-2} + 2^{i-1} + 5\cdot2^{i+1} + 2^{i+4} = 107\cdot
  2^{i-2},$ as claimed.  This yields
  \begin{equation}
    \sum_{i=0}^\infty r_{3\cdot
      2^i}(G)(3^{-s})(2^{-s})^i=3^{-s}\left(\frac{28 + 2\cdot2^{-s} - 10
        \cdot 2^{-2s} + 2 \cdot 2^{-3s}}{1-2\cdot
        2^{-s}}\right).\label{equ:3.2^i}
  \end{equation}
  Combining~\eqref{equ:2^i} and \eqref{equ:3.2^i} yields the claimed
  expression.
\end{proof}

The results in \cite{NoWo76} do not indicate how to compute the zeta
function of $\SL_2(\lri)$ for extension rings $\lri$ of $\Z_2$.  In
view of our earlier computations it would be particularly interesting
to consider the case where $\lri$ is an unramified extension of
$\Z_2$.

\section{Explicit formulae for subgroups of quaternion groups
  $\SL_1(\Ldr)$}
\label{sec:skew_fields}

The aim of this section is to provide a setup for computing the zeta
function of the norm one group $\SL_1(\Ldr)$ of a central division
algebra $\Ldr$ over a $\mfp$-adic field~$\lfi$.  Our approach leads to
immediate consequences in the special case where the Schur index
$\ell$ of $\Ldr$ over $\lfi$ is a prime number.  Explicit formulae are
given for the zeta functions of norm one groups of non-split
quaternion algebras; see Theorem~\ref{thm:quaternions} below.

\subsection{General division algebras} \label{sec:skew_sub_1} Let
$\Ldr$ denote a central division algebra of Schur index $\ell \geq 2$
over a $\mfp$-adic field~$\lfi$.  Let $\lri$ denote the valuation ring
in $\lfi$, with maximal ideal $\mfp$, and let $\ldr$ denote the
maximal compact subring of $\Ldr$, with maximal ideal $\mfP$.  Write
$q$ and $p$ for the cardinality and characteristic of the residue
field of $\lri$.

We consider the compact $p$-adic analytic group $G := \SL_1(\Ldr) =
\SL_1(\ldr)$ of norm-$1$ elements in $\Ldr$ and its principal
congruence subgroups $G_m := \SL_1^m(\Ldr) = \SL_1(\Ldr) \cap (1 +
\mfP^m)$, $m \in \N$.  We remark that the resulting congruence
filtration of $G$ is a refinement of the filtration that one would get
from restriction of scalars to $\lri$ and defining congruence
subgroups in terms of $\mfp$; this justifies the slight difference in
notation from Section~\ref{sec:p-adic_integral}.  The group $G/G_1$ is
isomorphic to the multiplicative group of norm-$1$ elements of
$\ldr/\mfP \cong \F_{q^\ell}$ over $\lri/\mfp \cong \F_q$, and hence
cyclic of order $(q^\ell -1)/(q-1)$.  Each of the quotients
$G_m/G_{m+1}$, $m \in \N$, embeds into the additive group $\ldr / \mfP
\cong \F_{q^\ell}$ and is thus an elementary abelian $p$-group.  It
follows that $G_1$ is the unique Sylow pro-$p$ subgroup of~$G$.  The
$\lfi$-Lie algebra associated to the group $\SL_1(\Ldr)$ is
$\spl_1(\Ldr)$, consisting of all trace-$0$ elements of $\Ldr$.

To begin with, we derive the following consequence of
Theorem~\ref{thm:abscissa_estimate} and its
Corollary~\ref{cor:abscissa_semisimple}, which is based on the extra
assumption that $\ell$ is prime.

\begin{cor} \label{cor:abscissa_skew} Let $\Ldr$ be a central division
  algebra of prime Schur index $\ell$ over $\lfi$.  Then the abscissa
  of convergence of $\zeta_{\SL_1(\Ldr)}(s)$ is equal to $2/\ell =
  2r_\textup{abs} / \lvert \Phi_\textup{abs} \rvert$, where
  $r_\textup{abs}$ is the absolute rank of $\spl_1(\Ldr)$ and
  $\Phi_\textup{abs}$ denotes the absolute root system associated to
  $\spl_1(\Ldr)$.
\end{cor}

We remark that a more complex argument shows that the conclusion of
the corollary remains true even if the Schur index $\ell$ is not
prime; see~\cite[Theorem~7.1]{LaLu08}.

\begin{proof}[Proof of Corollary~\ref{cor:abscissa_skew}]
  The absolute rank of $\spl_1(\Ldr)$ is $r_\textup{abs} = \ell-1$,
  and the size of the absolute root system associated to
  $\spl_1(\Ldr)$ is $\lvert \Phi_\textup{abs} \rvert = \ell^2 - \ell$.
  Hence $2/\ell = 2r_\textup{abs} / \lvert \Phi_\textup{abs} \rvert$.
  Clearly, $\dim_\lfi(\spl_1(\Ldr)) = \ell^2 -1$.  Recall that the
  abscissa of convergence $\alpha(\SL_1(\Ldr))$ is a commensurability
  invariant of the group $\SL_1(\Ldr)$.  Thus, in view of
  Theorem~\ref{thm:abscissa_estimate} and the equations
  \eqref{equ:sigmarho_centraliser}, it suffices to prove that
  $\dim_\lfi \Cen_{\spl_1(\Ldr)}(x) = \ell -1$ for all $x \in
  \spl_1(\Ldr) \setminus \{0\}$.  Indeed, this will imply
  $\sigma(\spl_1(\ldr)) = \rho(\spl_1(\ldr)) = \ell (\ell-1)/2$ and
  the result follows.

  Let $x \in \spl_1(\Ldr) \setminus \{0\}$.  Then $\lfi(x) \vert \lfi$
  is a non-trivial field extension, and the Centraliser Theorem for
  central simple algebras yields $\ell^2 = \lvert \Ldr : \lfi \rvert =
  \lvert \lfi(x) : \lfi \rvert \lvert \Cen_\Ldr(x) : \lfi \rvert$.
  Since $\ell$ is assumed to be prime, this implies $\lvert
  \Cen_\Ldr(x) : \lfi \rvert = \ell$, and hence $\dim_\lfi
  \Cen_{\spl_1(\Ldr)}(x) = \ell -1$.
\end{proof}

As a step toward the explicit computation of the zeta functions of
$G_1$ and $G$, we describe sufficient conditions for applying the
Kirillov orbit method to capture the irreducible complex characters of
the group $G_1$.

\begin{pro}\label{pro:div_ring_sat_pot}
  Let $\Ldr$ be a central division algebra of Schur index $\ell$ over
  the $\mfp$-adic field $\lfi$.  Let $\ldr$ be the maximal compact
  subring of $\Ldr$, and suppose that $p \nmid \ell$, where $p$
  denotes the residue field characteristic of $\lfi$.

  Then $G_1 =\SL_1^1(\ldr)$ is an insoluble maximal $p$-adic analytic
  just infinite pro-$p$ group.  Furthermore, if $e(\lfi,\Q_p)
  \ell < p-1$, then $G_1$ is potent and saturable.
\end{pro}

\begin{proof}
  It is known that $\Aut(\spl_1(\Ldr)) \cong \PGL_1(\Ldr)$;
  see~\cite[\S~XI~e]{KlLePl97}.  The groups $\SL_1(\Ldr)$ and
  $\PGL_1(\Ldr)$ are, of course, closely related.  Indeed, $Z :=
  \textup{Z}(\GL_1(\Ldr)) = \lfi^*$ and the norm map induces an
  isomorphism
  $$
  \GL_1(\Ldr) / Z \cdot \SL_1(\Ldr) \cong \lfi^* / (\lfi^*)^\ell.
  $$
  Denote by $\mu_\ell(\lfi^*)$ the finite subgroup of $\lfi^*$
  consisting of all elements whose order divides~$\ell$.  Since $p
  \nmid \ell$, reduction modulo $\mfp$ maps $\mu_\ell(\lfi^*) \leq
  \lri^*$ injectively onto the cyclic subgroup of order
  $\gcd(\ell,q-1)$ of $\F_q^*$.  Moreover, $G_1 \cap Z = G_1 \cap
  \mu_\ell(\lfi^*) = 1$, and the order of $\lfi^*/(\lfi^*)^\ell$ is
  not divisible by $p$.  Hence we see from the exact sequence
  $$
  1 \rightarrow \mu_\ell(\lfi^*) \rightarrow \SL_1(\Ldr) \rightarrow
  \PGL_1(\Ldr) \rightarrow \lfi^*/(\lfi^*)^\ell \rightarrow 1.
  $$
  that $G_1$ is isomorphic to a Sylow pro-$p$ subgroup of the compact
  group $\PGL_1(\Ldr)$.  It follows that $G_1$ is an insoluble maximal
  $p$-adic analytic just infinite pro-$p$ group;
  cf.~\cite[\S~III~e]{KlLePl97}.

  Let $\Lfi$ be a splitting subfield of $\Ldr$, unramified and of
  degree $\ell$ over $\lfi$.  Then the $\Lfi$-algebra isomorphism
  $\Lfi \otimes_\lfi \Ldr \cong \Mat_\ell(\Lfi)$ provides an embedding
  of $G_1$ into a Sylow pro-$p$ subgroup $S$ of $\SL_\ell(\Lri)$,
  where $\Lri$ denotes the valuation ring of $\Lfi$.  Suppose that $e
  \ell < p-1$, where $e = e(\lfi,\Q_p) = e(\Lri,\Z_p)$.
  From \cite[III (3.2.7)]{La65} we conclude that $S$ is saturable.
  Now we conclude as in~\cite[Proof of Theorem 1.3]{Kl05} that $G_1$
  is saturable.  Moreover, $\gamma_{p-1}(G_1) \subseteq \gamma_{e \ell
    + 1}(G_1) = G_1^p$ (cf.\ \cite[\S~1]{PrRa88}) so that $G_1$ is
  potent.
\end{proof}

%
%

\subsection{The quaternion case} \label{sec:quaternion_case} In this
section we consider the special case $\ell = 2$, where $\Ldr$ is a
non-split quaternion algebra over the $\mfp$-adic field $\lfi$, but
some of the arguments below are equally relevant in the more general
situation where $p \nmid \ell$.  Concretely, we compute explicit
formulae for the zeta functions of norm one quaternion groups
$\SL_1(\Ldr) = \SL_1(\ldr)$ and their principal congruence subgroups
$\SL_1^m(\ldr)$, as stated in Theorem~\ref{thm:quaternions} in the
introduction.



\begin{rem*}
  Similar formulae as the ones provided in
  Theorem~\ref{thm:quaternions} can be obtained for higher principal
  congruence subgroups $\SL_1^m(\ldr)$, even if the condition
  $e(\lfi,\Q_p) \ell < p-1$ imposed in the theorem is not satisfied.
  Our computation can be carried out in a similar fashion whenever the
  orbit method can be applied.
\end{rem*}

It is a natural and interesting problem to compute explicit formulae
for the zeta functions of norm one groups $\SL_1(\Ldr)$ and their
principal congruence groups, where the Schur index of $\Ldr$ over
$\lfi$ is greater than $2$.  Another interesting group to consider
would be $\SL_2(\ldr)$, where $\ldr$ is the maximal compact subring of
a non-split quaternion algebra $\Ldr$ over $\lfi$.

\begin{proof}[Proof of Theorem~\ref{thm:quaternions}]
  According to Proposition~\ref{pro:div_ring_sat_pot}, the pro-$p$
  group $G_1$ is saturable and potent.  Our first aim is to compute an
  explicit formula for the zeta functions of the principal congruence
  subgroups $G_m = \SL_1^m(\ldr)$, $m \in \N$.  Put $\mfg :=
  \spl_1(\ldr)$, and for $m \in \N$ let $\mfg_m := \spl_1^m(\ldr) =
  \spl_1(\ldr) \cap \mfP^m$ denote the $m$th principal congruence Lie
  sublattice so that $G_m = \exp(\mfg_m)$.  Here $\lri$ denotes, as
  usual, the valuation ring of $\lfi$, with maximal ideal $\mfp = \pi
  \lri$ generated by a uniformiser $\pi$.

  Since $p \not = 2$, we may choose $1$, $\mathbf{u}, \mathbf{v}$,
  $\mathbf{u} \mathbf{v}$ as a standard basis for $\Ldr$ over $\lfi$,
  where $\mathbf{u}^2 = a \in \lri$ is not a square modulo $\mfp$,
  $\mathbf{v}^2 = \pi$ and $\mathbf{u} \mathbf{v} = - \mathbf{v}
  \mathbf{u}$.  Writing $\mathbf{i} := \frac{1}{2} \mathbf{u}$,
  $\mathbf{j} := \frac{1}{2} \mathbf{v}$ and $\mathbf{k} :=
  \frac{1}{2} \mathbf{u} \mathbf{v}$ we have
  \begin{equation}\label{equ:ijk}
    [\mathbf{i},\mathbf{j}] = \mathbf{k}, \quad [\mathbf{i},\mathbf{k}]
    = a \mathbf{j}, \quad [\mathbf{j},\mathbf{k}] = -\pi \mathbf{i}.
  \end{equation}
  An $\lri$-basis for $\mfg_1$ is then given by $\pi \mathbf{i}$,
  $\mathbf{j}$, $\mathbf{k}$, with corresponding commutator matrix
  \begin{equation}\label{comm_matr_sl1}
    \mathcal{R}(\mathbf{Y}) =
    \begin{pmatrix}
      0    & \pi Y_3   & a \pi Y_2 \\
      -\pi Y_3 & 0     & -Y_1  \\
      -a \pi Y_2 & Y_1 & 0
    \end{pmatrix}.
  \end{equation}
  In view of Proposition~\ref{pro:zeta=poincare}, an argument similar
  as for the case $p=2$ and $e=1$ in Section~\ref{sec:expl_form_sl2}
  shows that
  $$
  \zeta_{\SL_1^m(\ldr)}(s) = q^{3(m-1)} \frac{q^2 -
    q^{-s}}{1-q^{1-s}} \quad \text{for $m \in \N$.}
  $$

  Our next aim is to deduce a formula for the zeta function of the
  group $G = \SL_1(\ldr)$, using Clifford theory, similarly as in
  Section~\ref{sec:full_sl_2}.  For this it is useful to record the
  following intermediate formula which comes from a similar argument
  as in Section~\ref{sec:expl_form_sl2}:
  \begin{equation} \label{equ:step_towards_sl1}
    \zeta_{\SL_1^1(\ldr)}(s) = 1 + \left( \mu(W(\lri)^{[1]}) q^{1-s} +
      \mu(W(\lri)^{[2]}) q^3 \right) \frac{1}{1-q^{1-s}},
  \end{equation}
  where $W(\lri)^{[1]} = \{ \mathbf{y} \in W(\lri) \mid y_1 \in \lri^*
  \}$ and $W(\lri)^{[2]} = \{ \mathbf{y} \in W(\lri) \mid y_1 \in \mfp
  \}$ have Haar measure $\mu(W(\lri)^{[1]}) = 1 - q^{-1}$ and
  $\mu(W(\lri)^{[2]}) = q^{-1}(1-q^{-2})$ respectively.

  We continue to write $G = \SL_1(\ldr)$ and put $N := G_1 =
  \SL_1^1(\ldr)$.  Since $G/N$ is cyclic, any irreducible character
  $\theta$ of $N$ can a priori be extended to an irreducible character
  $\hat \theta$ of its inertia group $I_G(\theta)$.  Thus, similarly
  as for the group $\SL_2(\lri)$, the central task consists in
  describing the inertia groups in order to apply
  \eqref{equ:Clifford}.  We show below that
  \begin{itemize}
  \item[(i)] $I_G(\theta) = G$, and hence $I_G(\theta)/N$ is cyclic
    of order $q+1$, if $\theta \in \Irr(G)$ corresponds to a
    co-adjoint orbit of a functional represented by an element of
    $W(\lri)^{[1]}$,
  \item[(ii)] $I_G(\theta) = \{1,-1\}N$, and hence $I_G(\theta)/N$ is
    cyclic of order $2$, if $\theta \in \Irr(G)$ corresponds to a
    co-adjoint orbit of a functional represented by an element of
    $W(\lri)^{[2]}$.
  \end{itemize}
  We conclude the proof by applying Clifford theory as in
  Section~\ref{sec:full_sl_2}. Putting together the general
  formula~\eqref{equ:Clifford}, the specific
  formula~\eqref{equ:step_towards_sl1} and statements~(i) and (ii), we
  deduce that
  \begin{align*}
    \zeta_{\SL_1(\ldr)}(s) & = \zeta_{C_{q+1}}(s) + \left(
      \mu(W(\lri)^{[1]})  q^{1-s} (q+1) + \right. \\
    & \qquad \left.  \mu(W(\lri)^{(2)}) q^3 ( (q+1)/2)^{-1-s} 2)
    \right) \frac{1}{1-q^{1-s}} \\
    & = \frac{(q+1) (1 - q^{-s}) + 4 (q-1) ((q+1)/2)^{-s}}{1-q^{1-s}}.
  \end{align*}

  It remains to justify the assertions (i) and (ii).  For this it is
  convenient to translate between the adjoint action of $G$ on $\mfg$
  and the co-adjoint action of $G$ on $\Hom_\lri(\mfg,\lri)$.  From
  \eqref{equ:ijk} one easily sees that the normalised Killing form
  $\kappa_0: \mfg \times \mfg \rightarrow \lri$ has the structure
  matrix
  $$
  [ \kappa_0( \cdot, \cdot)]_{(\mathbf{i},\mathbf{j},\mathbf{k})} =
  \left(
    \begin{smallmatrix}
      a & 0 & 0 \\
      0 & \pi & 0 \\
      0 & 0 & - a \pi
    \end{smallmatrix}
  \right)
  $$
  with respect to the basis $\mathbf{i}$, $\mathbf{j}$, $\mathbf{k}$.
  While $\kappa_0$ is degenerate over $\lri$, the form is
  non-degenerate over $\lfi$ and induces a bijective linear map
  $\iota_0: \spl_1(\Ldr) \rightarrow \Hom_\lfi(\spl_1(\Ldr),\lfi)$.
  We have $\mfg_{-1} := \spl_1(\ldr) \cap \mfP^{-1} =
  \iota_0^{-1}(\Hom_\lri(\mfg,\lri))$ and $\mfg_{-2} := \spl_1(\ldr)
  \cap \mfP^{-2} = \iota_0^{-1}(\Hom_\lri(\mfg_1,\lri))$.  The
  decomposition $\Hom_\lri(\mfg_1,\lri)^* = W(\lri)^{[1]} \, \dotcup
  \, W(\lri)^{[2]}$ corresponds to the decomposition 
  \begin{equation} \label{equ:g-2deco} 
    \mfg_{-2}^* = (\mfg_{-2}
    \setminus \mfg_{-1}) \, \dotcup \, (\mfg_{-1} \setminus \mfg).
  \end{equation}

  The action of $G/N$ on quotients $\mfg_m / \mfg_{m+1}$ of successive
  terms in the congruence filtration of $\mfg$ is described in
  \cite[\S~1]{PrRa88} and we will use a compatible notation as far as
  practical.  The division algebra $\Ldr$ contains an unramified
  extension $\Lfi = \lfi(\mathbf{i})$ of degree $2$ over $\lfi$ which
  is normalised by the uniformiser~$\Pi := \mathbf{j}$.  Let $F$
  denote the residue field of $\Lfi$, and let $\Phi$ denote the group
  of roots of unity in $\Lfi$.  Thus $F \cong \F_{q^2}$ and $\Phi \cup
  \{0\}$ is a set of representatives for the elements of $F$.  Observe
  that $N$ is complemented in $G$ by the subgroup $H$ of $\Phi$
  consisting of all roots of unity which are of norm $1$ over $\lfi$:
  we have $G = H \ltimes N$.  Accordingly, we will think of $G/N \cong
  H$ as the group of elements in the finite field $F$ which have norm
  $1$ over $f := \lri/\mfp \cong \F_q$.  Every element of $\ldr$ has a
  unique power series expansion in $\Pi$ with coefficients in $\Phi
  \cup \{0\}$.  For each $m \in \N$ this induces an embedding $\eta_m
  : \mfg_m / \mfg_{m+1} \hookrightarrow F$; we denote the image of
  $\eta_m$ by $F(m)$.  If $2 \nmid m$ then $F(m) = F$, and if $2 \mid
  m$ then $F(m) = \{ x \in F \mid \Tr_{F \vert f}(x) = 0 \}$.

  Clearly, for each $m \in \N$ the action of $G$ on $F(m)$ by
  conjugation factors through $N$ and is therefore determined by the
  action of $H$.  The latter is given by the explicit formula
  $$
  x^h = h^{1-q^m} \cdot x, \qquad \text{for $x \in F(m)$ and $h \in
    H$}.
  $$
  This finishes our preparations and we turn to the proof of
  assertions (i) and (ii) above.
 
  First we consider a character $\theta \in \Irr(N)$ corresponding to
  the co-adjoint orbit of $\omega \in \Irr(\mfg_1)$, where $\omega$ is
  represented by an element of $W(\lri)^{[1]}$ and has level $n$, say;
  cf.\ Section~\ref{subsec:basic_set_up}.  Then the inertia group
  $I_G(\theta)$ is equal to $CN$, where $C := \Cen_G(x + \mfg_{2n-2})$
  for a suitable $x \in \mfg_{-2} \setminus \mfg_{-1}$; see
  \eqref{equ:g-2deco}.  Here $\mfg_0 := \mfg$ if $n=1$.  We claim that
  $CN = G$, justifying~(i).  For this it is enough to prove that $H$
  centralises a suitable $N$-conjugate of $x$.  Since $H$ is a
  subgroup of the multiplicative group of the field $\Lfi$, it
  suffices to show that $x$ is $N$-conjugate to an element of $\Lfi$.
  For this we construct recursively a sequence $x_0$, $x_1$, \ldots of
  $N$-conjugates of $x$ such that $x_i \equiv \lambda_i \pi^{-1}
  \mathbf{i}$ modulo $\mfg_{i-1}$ with $\lambda_i \in \lri \setminus
  \mfp$ for each index $i$.  From the construction one sees that the
  sequence converges and its limit is an $N$-conjugate of $x$ in
  $\Lfi$.  Since $x \in \mfg_{-2} \setminus \mfg_{-1}$, we can take
  $x_0 := x$.  Now suppose that $i \in \N_0$ and that
  $$
  x_i = \lambda_i \pi^{-1} \mathbf{i} + \mu \mathbf{j} + \nu
  \mathbf{k} \qquad \text{with $\lambda_i, \mu, \nu \in \lri$,
    $\lambda_i \not\in \mfp$ and $\mu \mathbf{j} + \nu
  \mathbf{k} \in \mfP^{i-1}$}
  $$
  is an $N$-conjugate of $x$.  Then $x_{i+1} := z^{-1} x_i z$ with $z
  := 1 - \lambda_i^{-1} \pi (\nu \mathbf{j} - a^{-1} \mu \mathbf{k})
  \in 1 + \mfP^{i+1} \subseteq N$ is an $N$-conjugate of $x$ and
  satisfies the desired congruence, modulo $\mfg_i$,
  \begin{align*}
    x_{i+1} & \equiv (1 + \lambda_i^{-1} \pi (\nu \mathbf{j} + a^{-1}
    \mu \mathbf{k})) (\lambda_i \pi^{-1} \mathbf{i} + (\mu \mathbf{j}
    + \nu \mathbf{k})) (1 - \lambda_i^{-1} \pi (\nu \mathbf{j} +
    a^{-1} \mu \mathbf{k})) \\
    & \equiv \lambda_i \pi^{-1} \mathbf{i} + (\mu \mathbf{j} + \nu
    \mathbf{k}) + [\nu \mathbf{j} + a^{-1} \mu \mathbf{k}, \mathbf{i}] \\
    & \equiv \lambda_i \pi^{-1} \mathbf{i}.
  \end{align*}

  Finally we consider a character $\theta \in \Irr(N)$ corresponding
  to the co-adjoint orbit of $\omega \in \Irr(\mfg_1)$, where $\omega$
  is represented by an element of $W(\lri)^{[2]}$ and has level $n$,
  say.  Then the inertia group $I_G(\theta)$ is equal to $CN$, where
  $C := \Cen_G(x + \mfg_{2n-2})$ for a suitable $x \in \mfg_{-1}
  \setminus \mfg$; see \eqref{equ:g-2deco}.  We claim that $CN =
  \{1,-1\} N$, justifying~(ii).  Since $x \in \mfg_{-1}$, we have
  $\{1,-1\} N \subseteq CN \subseteq \Cen_G(x + \mfg) = \Cen_H(x +
  \mfg) N$.  Hence it suffices to prove that $\Cen_H(x + \mfg) =
  \{1,-1\}$.  Indeed, multiplication by $\pi$ provides an
  $H$-equivariant isomorphism $\mfg_{-1}/\mfg \rightarrow
  \mfg_1/\mfg_2$.  The action of $h \in H$ by conjugation on
  $\mfg_1/\mfg_2$ corresponds to multiplication by $h^{1-q}$ on
  $F(1)$; to carry out the multiplication $h$ is considered as a
  element of the residue field $F$ with norm $1$ in $f$, as described
  above.  The group $H$ has order $q+1$, hence the kernel of $H
  \rightarrow H$, $h \mapsto h^{1-q}$ is equal to $\{1,-1\}$.  It
  follows that $\Cen_H(\tilde x) = \{1,-1\}$ for any non-zero element
  $\tilde x \in F(1)$.
\end{proof}

\section{Principal congruence subgroups of $\SL_3(\lri)$ for
  unramified $\lri$ of residue field characteristic
  $3$}\label{sec:sl3_Z3}

Let $\lri$ be a compact discrete valuation ring of characteristic $0$
and residue field characteristic $p$.  Except for the special case $p
= 3$, Theorem~E in \cite{AvKlOnVo10+} provides an explicit universal
formula for the zeta functions of principal congruence subgroups of
$\SL_3(\lri)$.  By a different approach, the same formula and indeed a
formula for the group $\SL_3(\lri)$ itself are derived in
\cite{AvKlOnVo09}.  In this section we complement the generic formulae
by proving Theorem~\ref{thm:SL3_p=3} which was stated in the
introduction.  This theorem provides explicit formulae for the zeta
functions of principal congruence subgroups of $\SL_3(\lri)$, where
$\lri$ has residue characteristic $3$ and is unramified over $\Z_3$.


Residue field characteristic $3$ was excluded from
\cite[Theorem~E]{AvKlOnVo10+}, whose proof is based on a geometric
description of the variety of irregular elements in the
$8$-dimensional Lie algebra $\spl_3(\lfi)$.  This description breaks
down when the map $\beta: \gl_3 \rightarrow \spl_3$, $\mathbf{x}
\mapsto \mathbf{x} - \Tr(\mathbf{x})/3$ used in
\cite[Section~6.1]{AvKlOnVo10+} displays bad reduction modulo $\mfp$.
Moreover, the translation of the relevant $\mfp$-adic integral via the
normalised Killing form becomes more technical.  In the present paper
we restrict our attention to unramified extensions $\lri$ of $\Z_3$
for simplicity.  Indeed, the results in
Section~\ref{sec:expl_form_sl2} suggest that analogous formulae which
are to cover the general case, including ramification, are likely to
become rather cumbersome to write down.

The method we employ is algebraic and somewhat closer to the approach
taken in~\cite{AvKlOnVo09}.  In fact, the arguments which we shall
supply can be employed \emph{mutatis mutandis} in the generic case $p
\not = 3$ and hence give an alternative, less geometric derivation of
the formula provided in \cite[Theorem~E]{AvKlOnVo10+}.  Moreover, our
calculations illustrate the algebraic meaning of the $\mfp$-adic
formalism developed and applied in~\cite{AvKlOnVo10+}.

\subsection{} Let $\lri$ be unramified over $\Z_p$ with residue field
$\lri/\mfp \cong \F_q$ of characteristic $p=3$.  Throughout the
section we will continue to write $p$ as far as convenient, while
keeping the concrete value $p=3$ in mind.  We remark that $p$ is also
a uniformiser for $\lri$, because $\lri$ be unramified over $\Z_p$,
and we will use $p$ instead of the symbol $\pi$.  Let $T_\lri = \{0\}
\cup \mu_{q-1}(\lri)$ denote the set of Teichm\"uller representatives
for $\lri$, which projects bijectively onto the residue field of
$\lri$.  More generally, for any $l \in \N$ we fix
\begin{equation}\label{equ:Teichmueller}
  T_\lri(l) := \left\{ \sum\nolimits_{j=0}^{l-1} t_j p^j \mid
    t_j \in T_\lri \text{ for } 0 \leq j < l \right\}
\end{equation}
as a set of representatives for $\lri / \mfp^l$.  As usual, $\lfi$
denotes the field of fractions of~$\lri$.

Let $m \in \N$.  Proposition~2.3 in \cite{AvKlOnVo10+} shows that $m$
is permissible for $\spl_3(\lri)$.  Thus the zeta function of the
$m$th principal congruence subgroup $\SL_3^m(\lri)$ is given by the
formula
\begin{equation}\label{equ:still}
  \zeta_{\SL_3^m(\lri)}(s) =
  q^{8m} \left( 1 + (1-q^{-1})^{-1} \mathcal{Z}_\lri(-s/2-1,3s-3) \right),
\end{equation}
which results from Proposition~\ref{pro:zeta=poincare} on setting $d =
\dim_\lfi \spl_3(\lfi) = 8$ and $\rho = 2^{-1} (\dim_\lfi \spl_3(\lfi)
- r_\text{abs}(\spl_3(\lfi))) = 3$;
see~\eqref{equ:sigmarho_centraliser}.  As indicated in
Section~\ref{subsec:basic_set_up}, the integral
$\mathcal{Z}_\lri(r,t)$ is intimately linked to the elementary
divisors of the commutator matrix $\mathcal{R}(\mathbf{y})$ for
$\spl_3(\lri)$, evaluated at points $\mathbf{y} \in W(\lri) = \left(
\lri^8 \right)^*$.  We observe that the commutator matrix
$\mathcal{R}(\mathbf{y})$, at $\mathbf{y} \in W(\lri)$, has Witt
normal form
$$
\left(
\begin{smallmatrix}
  0 & 1 & & & & & & \\
 -1 & 0 & & & & & & \\
    &   & 0 & 1 & & & & \\
    &   &-1 & 0 & & & & \\
    &   &   &   & 0  & p^a & & \\
    &   &   &   &-p^a&   0 & & \\
    &   &   &   &    &     & 0 & 0 \\
    &   &   &   &    &     & 0 & 0
\end{smallmatrix}
\right)
$$
so that all the information is condensed in a single parameter $a =
a(\mathbf{y}) \in \N_0 \cup \{\infty\}$.  

The integral $\mathcal{Z}_\lri(r,t)$ in \eqref{equ:integral_neu} is
defined so that it performs integration over the space $\mfp \times
\Hom_\lri(\spl_3(\lri),\lri)^*$ with respect to a particular choice of
coordinate system $(x,\mathbf{y}) \in \mfp \times W(\lri)$.  The
normalised Killing form $\kappa_0$ of $\spl_3(\lfi)$ is related to
the ordinary Killing form $\kappa: \spl_3(\lfi) \times \spl_3(\lfi)
\rightarrow \lfi$ by the equation $\kappa = 2 h^\vee \kappa_0 = 6
\kappa_0$; see Section~\ref{sec:adjointversuscoadjoint}.  In
\cite[Section~6]{AvKlOnVo10+}, we provided the structure matrix of the
normalised Killing form $\kappa_0$ with respect to the basis
\begin{gather}
  \mathbf{h}_{12} = \left(
  \begin{smallmatrix}
    1 & & \\
    & -1 & \\
    & & 0
  \end{smallmatrix} \right), \quad \mathbf{h}_{23} = \left(
  \begin{smallmatrix}
    0 & & \\
    & 1 & \\
    & & -1
  \end{smallmatrix} \right), \nonumber \\
  \mathbf{e}_{12} = \left(
  \begin{smallmatrix}
    0 & 1 & \\
    & 0 & \\
    & & 0
  \end{smallmatrix} \right), \quad \mathbf{e}_{23} = \left(
  \begin{smallmatrix}
    0 &  & \\
    & 0 & 1 \\
    & & 0
  \end{smallmatrix} \right), \quad \mathbf{e}_{13} = \left(
  \begin{smallmatrix}
    0 & 0 & 1\\
    & 0 & 0 \\
    & & 0
  \end{smallmatrix} \right), \nonumber\\
  \mathbf{f}_{21} = \left(
  \begin{smallmatrix}
    0 & & \\
    1 & 0 & \\
    & & 0
  \end{smallmatrix} \right), \quad
  \mathbf{f}_{23} = \left(
  \begin{smallmatrix}
    0 &  & \\
    & 0 & \\
    & 1 & 0
  \end{smallmatrix} \right), \quad
  \mathbf{f}_{13} = \left(
  \begin{smallmatrix}
    0 & & \\
    0 & 0 & \\
    1 & 0 & 0
  \end{smallmatrix} \right). \nonumber
\end{gather}
This matrix has determinant $3$.  Thus the form $\kappa_0$ induces a
bijective linear map $\iota_0: \spl_3(\lfi) \rightarrow
\Hom_\lfi(\spl_3(\lfi),\lfi)$, but becomes more intricate at the level
of $\lri$-lattices, due to the residue field characteristic $3$.
Indeed, the pre-image of $\Hom_\lri(\spl_3(\lri),\lri)$ under
$\iota_0$ is the $\lri$-lattice
$$
\lat := \iota_0^{-1}(\Hom_\lri(\spl_3(\lri),\lri)) = \bigcup_{u \in
  T_\lri(1)} \left( u (\tfrac{2}{3} \mathbf{h}_{12} + \tfrac{1}{3}
  \mathbf{h}_{23}) + \spl_3(\lri) \right).
$$ Thus we have $p\lat \leq \spl_3(\lri) \leq \lat$ with $\lvert
\spl_3(\lri) : p\lat \rvert = q^7$ and $\lvert \lat : \spl_3(\lri)
\rvert = q$.  We pull back the integral $\mathcal{Z}_\lri(r,t)$ over
$\mfp \times (\Hom_\lri(\spl_3(\lri),\lri))^*$ to an integral over
$\mfp \times \lat^*$, taking into account the Jacobi factor $\vert 3
\vert_\mfp = q^{-1}$.  Dividing the new region of integration with
respect to the second factor into cosets modulo $p\lat$, we write
\begin{equation}\label{equ:division}
  \mathcal{Z}_\lri(r,t) = \mathcal{S}_1(r,t) + \mathcal{S}_2(r,t),
\end{equation}
where the two summands correspond to the complementary subregions of
integration $\mfp \times (\spl_3(\lri) \setminus p\lat)$ and $\mfp
\times (\lat \setminus \spl_3(\lri))$ respectively.  We show in
Sections~\ref{sec:summand_1} and \ref{sec:summand_2} that these
summands are given by the formulae
\begin{equation}\label{equ:S1_and_S2}
  \begin{split}
    \mathcal{S}_1(r,t) & = (q^4+q^3-q-1) \mathcal{Z}_\lri^{[3]}(r,t) +
    (q^6-q^4-q^3+q) \mathcal{Z}_\lri^{[2]}(r,t) + (q^7 - q^6)
    \mathcal{Z}_\lri^{[0]}(r,t), \\
    \mathcal{S}_2(r,t) & = (q-1) ((q^2+q+1)q^2
    \mathcal{Z}_\lri^{[1]}(r,t) + (q^7 - (q^2+q+1)q^2)
    \mathcal{Z}_\lri^{[0]}(r,t)),
  \end{split}
\end{equation}
where
\begin{align*}
  \mathcal{Z}_\lri^{[0]}(r,t) & = \int_{(x,\mathbf{y}) \in \mfp \times
    \mfp^{(8)}} \lvert x \rvert^t \, d\mu(x,\mathbf{y}), \\
  \mathcal{Z}_\lri^{[1]}(r,t) & = \int_{(x,\mathbf{y}) \in \mfp \times
    \mfp^{(8)}} \lvert x \rvert^t \lVert \{y_1,
  y_2, y_3,x\} \rVert_\mfp^{2r} \, d\mu(x,\mathbf{y}), \\
  \mathcal{Z}_\lri^{[2]}(r,t) & = q^{-2r}   \mathcal{Z}_\lri^{[0]}(r,t), \\
  \mathcal{Z}_\lri^{[3]}(r,t) & = \int_{(x,\mathbf{y}) \in \mfp \times
    \mfp^{(8)}} \lvert x \rvert^t \lVert \{p y_1, p y_2, y_3,x\}
  \rVert_\mfp^{2r} \, d\mu(x,\mathbf{y}).
\end{align*}

In \cite[Section~6.1]{AvKlOnVo10+} it is shown that
\begin{equation}\label{equ:Z0}
  \mathcal{Z}_\lri^{[0]}(r,t) = \frac{q^{-9-t}(1 - q^{-1})}{1 - q^{-1-t}}
\end{equation}
and
\begin{equation}\label{equ:Z1}
  \mathcal{Z}_\lri^{[1]}(r,t) = \frac{q^{-9-2r-t} (1 - q^{-4-t})(1 -
    q^{-1})}{(1 - q^{-4-2r-t})(1 - q^{-1-t})}.
\end{equation}
The latter may be obtained from the formula
$$
\mathcal{Z}_\lri^{[1]}(r,t) = \sum_{(l,n) \in \N^2} (1 - q^{-1})
q^{-n} \Mass^{[1]}_l \, q^{-nt - 2 \min \{l,n\} r},
$$ 
with
\begin{equation}\label{equ:def_Mass^1}
  \Mass^{[1]}_l := \mu \left( \left\{ \mathbf{y} \in \mfp^{(8)}
      \mid \max \{ \lvert y_1 \rvert_\mfp, \lvert
      y_2 \rvert_\mfp, \lvert y_3 \rvert_\mfp \}
      = q^{-l} \right\} \right) = (1 - q^{-3}) q^{-3l-5},
\end{equation}
using the fact that
\begin{equation}\label{equ:geometric_series}
  \sum_{(l,n) \in \N^2} X_1^l X_2^n X_3^{\min\{l,n\}} = \frac{X_1 X_2
    X_3 (1 - X_1 X_2)}{(1 - X_1 X_2 X_3)(1 - X_1)(1 - X_2)}.
\end{equation}

Clearly, \eqref{equ:Z0} implies that
\begin{equation}\label{equ:Z2}
\mathcal{Z}_\lri^{[2]}(r,t) = \frac{q^{-9-2r-t}(1 - q^{-1})}{1 - q^{-1-t}}.
\end{equation}
This formula, too, may be written as a sum
$$
\mathcal{Z}_\lri^{[2]}(r,t) = \sum_{(l,n) \in \N^2} (1 - q^{-1})
q^{-n} \Mass^{[2]}_l \, q^{-nt - 2 \min \{l,n\} r},
$$ with
\begin{equation}\label{equ:def_Mass^2}
  \Mass^{[2]}_l := \mu \left( \left\{ \mathbf{y} \in
      \mfp^{(8)} \mid \vert p \rvert_\mfp=q^{-l} \right\} \right) = 
  \begin{cases}
    q^{-8} &\text{ if }l=1,\\
    0 &\text{ if }l\geq 2.
  \end{cases}
\end{equation}

It remains to compute $\mathcal{Z}_\lri^{[3]}(r,t)$.  We have
$$ \mathcal{Z}_{\lri}^{[3]}(r,t) = \sum_{(l,n) \in \N^2} (1 - q^{-1})
q^{-n} \Mass^{[3]}_l \, q^{-nt - 2 \min \{l,n\}r},
$$ where
\begin{align}
  \Mass^{[3]}_l & := \mu \left( \left\{ \mathbf{y} \in \mfp^{(8)} \mid
      \max \{ \lvert p y_1 \rvert_\mfp, \lvert p y_2 \rvert_\mfp,
      \lvert y_3 \rvert_\mfp \} = q^{-l} \right\} \right) \nonumber
      \\&=
  \begin{cases}
    (1-q^{-1}) q^{-8} & \text{if $l=1$,} \\
    (1-q^{-3}) q^{-3-3l} & \text{if $l \geq 2$.}
  \end{cases}\label{def: M_l}
\end{align}
Using \eqref{equ:geometric_series} this gives
\begin{align*}
  \mathcal{Z}_\lri^{[3]}(r,t) & = (1 - q^{-1}) (\Mass^{[3]}_1 -
  (1-q^{-3}) q^{-6}) \sum_{n \in \N} q^{-(1+t)n - 2 r} \\ & \quad + (1
  - q^{-1}) (1-q^{-3}) q^{-3} \sum_{(l,n) \in \N^2} q^{-3l - (1+t)n -
  2\min \{l,n\}r} \\ & = - (1 - q^{-1}) (1 - q^{-2}) q^{-7-2r-t} (1 -
  q^{-1-t})^{-1} \\ & \quad + (1 - q^{-1}) (1- q^{-4-t}) q^{-7-t-2r}
  (1- q^{-1-t})^{-1} (1- q^{-4-t-2r})^{-1} \\ & = \frac{(1- q^{-1})
  q^{-9-t-2r} \left( 1 - q^{-2-t} + q^{-2-t-2r} - q^{-4-t-2r}
  \right)}{(1 - q^{-1-t}) (1 - q^{-4-t-2r})} .
\end{align*}
A short computation, based on \eqref{equ:still}, \eqref{equ:division}
and \eqref{equ:S1_and_S2}, now yields the explicit formula for the
zeta function of $\SL_3^m(\lri)$, stated in Theorem~\ref{thm:SL3_p=3}.

The remainder of this section is devoted to an algebraic justification
of the equations~\eqref{equ:S1_and_S2}.  It would be interesting to
derive a geometric explanation, more similar to the argument in
\cite[Section~6.1]{AvKlOnVo10+} treating the generic case.

\subsection{} \label{sec:summand_1} First we will derive the formula
given for the summand $\mathcal{S}_1(r,t)$ in~\eqref{equ:S1_and_S2}.
For this we decompose $\spl_3(\lri) \setminus p\lat$ into cosets
modulo $p\lat$, or equivalently the finite Lie algebra
$\spl_3(\mathbb{F}_q)$ into cosets modulo its centre.  As $p=3$, the
centre of $\spl_3(\mathbb{F}_q)$ is the $1$-dimensional subalgebra
$\mathbb{F}_q \overline{\mathbf{z}}$, spanned by the reduction modulo
$\mfp$ of $\mathbf{z} := \mathbf{h}_{12} - \mathbf{h}_{23}$, viz.\ the
subalgebra of scalar matrices over $\F_q$.  An overview of the orbits
in $\spl_3(\F_q)$ under the adjoint action of $\GL_3(\mathbb{F}_q)$ is
provided in Table~\ref{table_p=3}; see Appendix~\ref{sec:aux_sl3_p=3}
for a short discussion.  The second column indicates whether the
corresponding elements are regular or irregular, as defined at the end
of Section~\ref{subsec:basic_set_up}.  The corresponding total number
of cosets modulo $\F_q \overline{\mathbf{z}}$, given in the last
column of the table, is obtained upon division by~$q$.  A short
calculation confirms that the sizes of the orbits listed in
Table~\ref{table_p=3} add up to $q^7 = \lvert \spl_3(\F_q) : \F_q
\overline{\mathbf{z}} \rvert$, as wanted.

\begin{table}[H]
  \centering
  \begin{tabular}{|c|l|l|l|l|}
    \hline
    type &  & no.\ of orbits & size of each orbit & total
    number modulo $\F_q \overline{\mathbf{z}}$ \\
    \hline\hline
    \caseA{} 
&  & $q$ & $1$ & $1$ \\
    \caseC{} 
& irregular\ & $q$ & $(q^3-1)(q^2-1)q$ &
    $(q^3-1)(q^2-1)q$ \\ 
    \caseB{} 
& irregular\ & $q$ & $(q^3-1)(q+1)$ &
    $(q^3-1)(q+1)$ \\ 
    \caseF{} 
& regular\ & $(q-1)q/6$ &
    $(q^2+q+1)(q+1)q^3$ & $(q^3-1)(q+1)q^3/6$ \\
    \caseD{} 
& regular\ & $(q-1)q/2$ & $(q^3-1)q^3$ &
    $(q^3-1)(q-1)q^3/2$ \\ 
    \caseE{} 
& regular\ & $(q-1)q/3$ & $(q^2-1)(q-1)q^3$ &
    $(q^2-1)(q-1)^2q^3/3$ \\
    \hline
  \end{tabular}
  \smallskip
  \caption{Adjoint orbits in $\spl_3(\F_q)$ under the action of
    $\GL_3(\F_q)$, $q \equiv_3 0$}
  \label{table_p=3}
\end{table}

The equation for $\mathcal{S}_1(r,t)$ in \eqref{equ:S1_and_S2}
indicates that $\mathcal{Z}_\lri^{[0]}(r,t)$ is the correct integral for
the types \caseF{}, \caseD{}, \caseE{}, which cover the regular
elements modulo $p$.  It remains to link the contributions to the
summand $\mathcal{S}_1(r,t)$ by irregular elements belonging to cosets
of types \caseC{} and \caseB{} to the integrals
$\mathcal{Z}_\lri^{[2]}(r,t)$ and $\mathcal{Z}_\lri^{[3]}(r,t)$,
respectively.

\subsubsection{} Let us consider first elements belonging to cosets
modulo $p\lat$ of type~\caseB{} and work out the integral around such
elements which results from pulling the original integral
$\mathcal{Z}_\lri(r,t)$ over $\mfp \times
(\Hom_\lri(\spl_3(\lri),\lri))^*$ back to $\mfp \times \lat^*$.  A
typical coset of type~\caseB{} is $\mathbf{a}_0 + p\lat$, where
$\mathbf{a}_0 := \left( \begin{smallmatrix} 0 & 1 & 0 \\ 0 & 0 & 0 \\
    0 & 0 & 0
  \end{smallmatrix} \right)$, and each coset has measure $\mu(p\lat) =
q^{-7}$.  As indicated earlier, the determinant of the Jacobi matrix
associated to $\iota_0 : \lat \rightarrow
\Hom_\lri(\spl_3(\lri),\lri)$ is $3$ and thus contributes another
factor $\lvert 3 \rvert_\mfp = q^{-1}$.  The integral over $\mfp
\times (\mathbf{a} + p\lat)$ with Jacobi factor $q^{-1}$ can thus be
described as an integral $\mathcal{I}(r,t)$ over $\mfp \times
\mfp^{(8)}$.  We argue that it is equal to
$\mathcal{Z}_\lri^{[3]}(r,t)$, which may be computed from the integer
sequence $\Anzahl^{[3]}_n$, $n \in\N_0$, defined by
\begin{align*}
  \Anzahl^{[3]}_n & := \# \left\{\bfy + (\mfp^{n+1})^{(8)} \mid \bfy
    \in \mfp^{(8)} \text{ such that } \lVert \{ p y_1, p y_2, y_3,
    p^{n+1} \} \rVert _{\mfp}=q^{-n-1} \right\} \\ & \phantom{:}= \#
    \left\{ \bfy + (\mfp^{n+1})^{(8)} \mid \bfy \in\mfp^{(8)} \text{
    such that } y_1, y_2 \in \mfp^n \text{ and } y_3 \in \mfp^{n+1}
    \right\}\\ & \phantom{:}= \lvert \mfp^{\max\{1,n\}} : \mfp^{n+1}
    \rvert^2 \cdot \vert \mfp^{(5)} : (\mfp^{n+1})^{(5)} \vert \\ &
    \phantom{:}=
  \begin{cases} 
    1 & \text{ if } n=0,\\
    q^{5n+2} & \text{ if }n\geq 1,
  \end{cases}
\end{align*}
describing the lifting behaviour of points modulo $\mfp^{n+1}$ on the
variety defined by the integrand of
$\mathcal{Z}_\lri^{[3]}(r,t)$. Indeed, we observe that the numbers
$\Mass^{[3]}_l$ defined in~\eqref{def: M_l} satisfy
\begin{equation*}
\Mass^{[3]}_l = q^{-8l}\Anzahl^{[3]}_{l-1} - q^{-8(l+1)}\Anzahl^{[3]}_l.
\end{equation*}
The following proposition shows that the integral $\mathcal{I}(r,t)$
over $\mfp \times \mfp^{(8)}$ is equal to
$\mathcal{Z}_\lri^{[3]}(r,t)$.

\begin{pro} \label{pro:schwierigster_fall}
 For $n \in \N_0$ the set
  \begin{multline*}
    \Menge^{[3]}_n := \left\{ \mathbf{a} + p^{n+1}\lat \in \lat/p^{n+1}\lat \mid
      \mathbf{a} \equiv \left( \begin{smallmatrix} 0 & 1 & 0 \\ 0 & 0
          & 0 \\ 0 & 0 & 0 \end{smallmatrix} \right) \textnormal{ modulo
        $p\lat$} \right. \\ \left. \textnormal{and } \lvert \spl_3(\lri) :
      \Cen_{\spl_3(\lri)} (\mathbf{a} + p^{n+1}\lat) \rvert = q^{4(n+1)}
    \right\}
  \end{multline*}
has cardinality $\Anzahl_n^{[3]}$.
\end{pro}

\begin{proof}
  The case $n=0$ is a simple computation.  Indeed, the only candidate
  for an element of $\Menge^{[3]}_0$ is $\mathbf{a}_0 + p\lat$, where
  $\mathbf{a}_0 :=
  \left( \begin{smallmatrix} 0 & 1 & 0 \\
      0 & 0 & 0 \\ 0 & 0 & 0 \end{smallmatrix} \right)$, and a short
  computation reveals that, indeed,
  $$
  \lvert \spl_3(\lri) : \Cen_{\spl_3(\lri)} (\mathbf{a}_0 + p\lat) \rvert =
  \lvert \spl_3(\F_q) : \Cen_{\spl_3(\F_q)} (\overline{\mathbf{a}_0})
  \rvert = q^{4},
  $$ where $\overline{\mathbf{a}_0}$ denotes the image of
  $\mathbf{a}_0$ in $\spl_3(\F_q)$; cf.\ \eqref{equ:centr_type_2}.
  Thus $\lvert \Menge^{[3]}_0 \rvert = 1$, as claimed.

  Now suppose that $n \geq 1$.  Arguing by induction on $n$, we prove
  in fact a little more than stated in the proposition.  For any $l
  \in \N$, let $T_\lri(l)$ denote representatives for $\lri / \mfp^l$
  derived from the Teich\-m\"uller representatives for $\lri/\mfp$;
  see \eqref{equ:Teichmueller}.

  \smallskip

  \noindent
  \emph{Claim.} Every matrix $\mathbf{a} \in \lat$ with $\mathbf{a} +
  p^{n+1}\lat \in \Menge^{[3]}_n$ can be conjugated by elements of
  $\GL_3^1(\lri)$ to the `normal' form
  $$
  \begin{pmatrix}
    0 & 1 & 0 \\
    2p^2 c^2 & p c & 0 \\
    0 & 0 & -pc
  \end{pmatrix}
  + p^n
  \begin{pmatrix}
    0 & 0 & 0 \\
    0 & 0 & y_3 \\
    y_2 & 0 & 0
  \end{pmatrix} 
  \qquad \text{modulo $p^{n+1}\lat$,}
  $$ where $c \in T_\lri(n)$ and $y_2, y_3 \in T_\lri(1)$. These
  matrices modulo $p^{n+1} \lat$ form a complete set of
  representatives for the $\GL_3^1(\lri)$-orbits comprising
  $\Menge^{[3]}_n$.  Moreover, the index in $\GL_3^1(\lri)$ of the
  centraliser of any such matrix modulo $p^{n+1}\lat$ is $q^{4n}$, and
  thus $\lvert \Menge^{[3]}_n \rvert = \lvert T_\lri(n) \rvert \lvert T_\lri(1)
  \rvert^2 q^{4n} = q^{5n+2}$, as wanted.

  Finally, every matrix $\mathbf{a} \in \lat$ with $\mathbf{a} +
  p^{n+2}\lat \in \Menge^{[3]}_{n+1}$ can be conjugated by elements of
  $\GL_3^1(\lri)$ to a matrix which is, modulo $p^{n+1} \lat$, of the
  normal form above and satisfies the extra condition $y_2 = y_3 = 0$.

  \smallskip

  As indicated we use induction on $n$.  Let $c \in T_\lri(n-1)$ and put
  $$
  \mathbf{a}_c :=
  \begin{pmatrix}
    0 & 1 & 0 \\
    2p^2 c^2 & p c & 0 \\
    0 & 0 & -p c
  \end{pmatrix}.
  $$ The eigenvalues of $\mathbf{a}_c$ are $-p c$ and $2p c$ with
  multiplicities $2$ and $1$ respectively.  Hence $c$ is an invariant
  of the $\GL_3^1(\lri)$-orbit of $\mathbf{a}_c$ modulo
  $\spl_3^n(\lri)$.  In view of our discussion of the case $n=0$, if
  $n=1$, or the induction hypothesis, if $n>1$, it suffices to work
  out representatives of the elements of $\Menge^{[3]}_n$ within the
  set $\mathbf{a}_c + \spl_3^n(\lri)$ modulo $\spl_3^{n+1}(\lri)$.  We
  consider the set $\mathbf{a}_c + \spl_3^n(\lri)$ modulo
  $\spl_3^{n+1}(\lri)$, up to conjugation by $\GL_3^n(\lri)$.  Let
  \begin{equation} \label{equ:generic_matrix} \mathbf{x} :=
    \begin{pmatrix}
      x_1 & x_2 & x_3 \\
      x_4 & x_5 & x_6 \\
      x_7 & x_8 & x_9
    \end{pmatrix}
    \in \Mat_3(\lri).
  \end{equation}
  Then
  \begin{align}
    \left[ \mathbf{a}_c, \mathbf{x} \right] & =
    \begin{pmatrix}
      x_4 - 2p^2 c^2 x_2 & (x_5 - x_1) - pc x_2 & x_6 + pc x_3 \\
      pc(x_4 + 2pc(x_1 - x_5)) & -x_4 + 2p^2 c^2 x_2 & 2pc(x_6 +
      pc x_3) \\
      -pc (x_7 + 2pc x_8) & -x_7-2pc x_8 & 0
    \end{pmatrix} \label{equ:commutator_complete} \\
    & \equiv_{\mfp}
    \begin{pmatrix}
      x_4 & x_5-x_1 & x_6 \\
      0 & -x_4 & 0 \\
      0 & -x_7 & 0
    \end{pmatrix}. \label{equ:commutator_congruence}
  \end{align}
  If $\mathbf{b} = \mathbf{a}_c + p^n \mathbf{y} \in \mathbf{a}_c +
  \spl_3^n(\lri)$ and $g = 1 + p^n \mathbf{x} \in \GL_3^n(\lri)$, then
  \begin{align}
    g^{-1} \mathbf{b} g & \equiv (1-p^n \mathbf{x})(\mathbf{a}_c + p^n
    \mathbf{y})(1+p^n \mathbf{x})\nonumber \\
    & \equiv \mathbf{a}_c + p^n (\mathbf{y} +
    [\mathbf{a}_c,\mathbf{x}])\label{equ:conjugation}
  \end{align}
  modulo $\spl_3^{n+1}(\lri)$.  In view of
  \eqref{equ:commutator_congruence} this shows that the elements
  $$
  \mathbf{b}_c(y_1,y_2,y_3,y_4) := \mathbf{a}_c + p^n \begin{pmatrix}
    0 & 0 & 0 \\
    y_1 & y_4 & y_3 \\
    y_2 & 0 & -y_4
    \end{pmatrix}
  $$
  with $y_1,y_2,y_3, y_4 \in T_\lri(1)$ form a complete set of
  representatives for the $\GL_3^n(\lri)$-orbits of $\mathbf{a}_c +
  \spl_3^n(\lri)$ modulo $\spl_3^{n+1}(\lri)$, and indeed modulo
  $p^{n+1} \lat$.

  Consider one of these lifts, $\mathbf{b} =
  \mathbf{b}_c(y_1,y_2,y_3,y_4)$.  In order to simplify the notation,
  it is convenient to use the fact that $\mathbf{b} =
  \mathbf{b}_{\tilde c}(y_1,y_2,y_3,0)$, where $\tilde c := c +
  p^{n-1} y_4 \in T_\lri(n)$, and to work with
  $$
  \mathbf{a}_{\tilde c} :=
  \begin{pmatrix}
    0 & 1 & 0 \\
    2p^2 \tilde c^2 & p \tilde c & 0 \\
    0 & 0 & -p \tilde c
  \end{pmatrix}
  $$
  instead of $\mathbf{a}_c$.  In order to describe the centraliser
  index of $\mathbf{b} + p^{n+1}\lat$ in $\spl_3(\lri)$, we consider
  again a generic matrix $\mathbf{x}$ as in
  \eqref{equ:generic_matrix}, now with the additional restriction that
  $\mathbf{x} \in \spl_3(\lri)$, viz.\ $x_1 + x_5 + x_9 = 0$.  One
  computes
  \begin{multline*}
    [\mathbf{b},\mathbf{x}] = [\mathbf{a}_{\tilde c},\mathbf{x}] + \\
    p^n
      \begin{pmatrix}
        -y_1 x_2 - y_2 x_3 & 0 & - y_3 x_2 \\
        y_1 (x_5-x_1) - y_2 x_6 + y_3 x_7 & y_1 x_2 + y_3 x_8 & y_1
        x_3 + y_3 (x_9-x_5) \\
        y_2 (x_1-x_9) - y_1 x_8 & y_2 x_2 & y_2 x_3 - y_3 x_8
      \end{pmatrix}.
  \end{multline*}
  Taking into account \eqref{equ:commutator_complete}, the condition
  $[\mathbf{b}, \mathbf{x}] \equiv 0$ modulo $p^{n+1} \lat$ can be
  expressed in terms of the following list of restrictions on the
  entries of $\mathbf{x}$, involving the parameters $y_1,y_2,y_3 \in
  T_\lri(1)$:
  \begin{enumerate}
  \item[(i)] $x_4 - 2 p^2 {\tilde c}^2 x_2 - p^n (y_1 x_2 + y_2 x_3)
    \equiv_{p^{n+1}} -x_4 + 2 p^2 {\tilde c}^2 x_2 + p^n (y_1 x_2 +
    y_3 x_8)$ from the $(1,1)$- and $(2,2)$-entries, equivalently $2x_4
    \equiv_{p^{n+1}} 4 p^2 {\tilde c}^2 x_2 + p^n (2y_1 x_2 + y_2 x_3 +
    y_3 x_8)$,
  \item[(ii)] $x_5 \equiv_{p^{n+1}} x_1 + p {\tilde c} x_2$ from
    the $(1,2)$-entry,
  \item[(iii)] $x_6 \equiv_{p^{n+1}} -p {\tilde c} x_3 + p^n y_3
    x_2$ from the $(1,3)$-entry,
  \item[(iv)] $x_7 \equiv_{p^{n+1}} -2 p {\tilde c} x_8 + p^n y_2
    x_2$ from the $(3,2)$-entry, \vspace{.2cm}
  \item[(v)] $0 \equiv_p - y_2 x_6 + y_3 x_7$ from the $(2,1)$-entry,
    but this condition becomes redundant if $x_6, x_7 \equiv_p 0$,
    \vspace{.2cm}
  \item[(vi)] $0 \equiv_p y_2 (x_1-x_9) - y_1 x_8 \equiv_p y_2 (x_1 +
    x_5 + x_9) - y_1 x_8 \equiv_p -y_1 x_8$ from the $(3,1)$-entry and
    (ii),
  \item[(vii)] $0 \equiv_p y_1 x_3 + y_3 (x_9-x_5) \equiv_p y_1 x_3 +
    y_3 (x_1 + x_5 + x_9) \equiv_p y_1 x_3$ from the $(2,3)$-entry and
    (ii),
  \item[(viii)] $-x_4 + 2 p^2 {\tilde c}^2 x_2 + p^n (y_1 x_2 + y_3 x_8)
    \equiv_{p^{n+1}} p^n (y_2 x_3 - y_3 x_8$) from the $(2,2)$- and
    $(3,3)$-entries, equivalently $x_4 \equiv_{p^{n+1}} 2 p^2 {\tilde c}^2
    x_2 + p^n (y_1 x_2 - y_2 x_3 + 2 y_3 x_8)$.
  \end{enumerate}
  If these conditions are to hold for $\mathbf{x} \in \spl_3(\lri)$,
  then the congruences (i)-(iv) show that the entries $x_4, x_5 , x_6,
  x_7$ are determined completely modulo $p^{n+1}$ by the remaining
  entries $x_1, x_2, x_3, x_8, x_9$.  Because $x_1 + x_5 + x_9 =
  \Tr(\mathbf{x}) = 0$, we can also think of $x_9$ as being determined
  by $x_1$.  We observe that $\lvert \spl_3(\lri) :
  \Cen_{\spl_3(\lri)} (\mathbf{b} + p^{n+1}\lat) \rvert = q^{4(n+1)}$
  (and not larger) if and only if one can choose $x_1, x_2, x_3, x_8$
  freely, i.e.\ if the remaining conditions (v)-(viii) do not impose
  extra restrictions.

  In fact, the congruence (v) will be automatically satisfied, as
  indicated, because $x_6, x_7 \equiv_p 0$.  As $x_1 + x_5 + x_9 =
  \Tr(\mathbf{x}) = 0$, the conditions (vi) and (vii) will give rise
  to restrictions on $x_8$ or $x_3$, unless $y_1 \equiv_p 0$.
  Finally, the last condition (viii) is equivalent to the condition
  (i), since $p=3$.

  The discussion so far shows that, with respect to the action of
  $\GL_3^n(\lri)$, the intersection of $\Menge^{[3]}_n$ and $(\mathbf{a}_c +
  \spl_3^n(\lri)) / p^{n+1}\lat$ consists of $q^3$ orbits, represented
  by matrices
  $$
  \mathbf{b}_c(0,y_2,y_3,y_4)
  $$ 
  and each of size $q^4$.  By induction, there are $q^{5(n-1)}$
  matrices modulo $p^{n-1} \lat$, represented by matrices such as
  $\mathbf{a}_c$, which lift to elements of $\Menge^{[3]}_n$.  Hence
  $\lvert \Menge^{[3]}_n \rvert = q^{5(n-1)} q^{3+4} = q^{5n+2}$.  In
  order to show that the $\lvert T_\lri(n-1) \rvert q^3 = q^{n+2}$
  matrices $\mathbf{b}_c(0,y_2,y_3,y_4)$ form a complete set of
  representatives for the $\GL_3^1(\lri)$-orbits comprising
  $\Menge^{[3]}_n$, it suffices to show that for any one of them,
  $\mathbf{b} = \mathbf{b}_c(0,y_2,y_3,y_4)$ say, one has
  \begin{equation} \label{equ:centr=q4n} \lvert \GL_3^1(\lri) :
    \Cen_{\GL_3^1(\lri)} (\bfb + p^{n+1}\lat) \rvert = q^{4n}.
  \end{equation}
  We make three observations.  Firstly, a straightforward translation
  between $\GL_3^1(\lri)$ and its Lie lattice $\gl_3^1(\lri)$ yields
  $$
  \lvert \GL_3^1(\lri) : \Cen_{\GL_3^1(\lri)} (\bfb + p^{n+1}\lat)
  \rvert = \lvert \gl_3^1(\lri) : \Cen_{\gl_3^1(\lri)} (\bfb +
  p^{n+1}\lat) \rvert.
  $$
  Secondly, we note that $\mathbf{c} :=
  \left( \begin{smallmatrix}0&&\\&0&\\&&p \end{smallmatrix} \right)
  \in \Cen_{\gl_3^1(\lri)}(\mathbf{b} + p^{n+1} \lat)$.  Since
  $\gl_3^1(\lri) = \lri \mathbf{c} \oplus \spl_3^1(\lri)$, this
  implies that $\gl_3^1(\lri) = \Cen_{\gl_3^1(\lri)}(\mathbf{b} +
  p^{n+1} \lat) + \spl_3^1(\lri)$, and consequently
  $$
  \lvert \gl_3^1(\lri) : \Cen_{\gl_3^1(\lri)} (\bfb + p^{n+1}\lat)
  \rvert = \lvert \spl_3^1(\lri) : \Cen_{\spl_3^1(\lri)} (\bfb +
  p^{n+1}\lat) \rvert.
  $$
  Finally, we observe that the property $\lvert \spl_3(\lri) :
  \Cen_{\spl_3(\lri)} (\bfb + p^{n+1}\lat) \rvert = q^{4(n+1)}$ is, in
  fact, equivalent to $\lvert \spl_3(\lri) : \Cen_{\spl_3(\lri)} (\bfb
  + p\lat) \rvert = q^4$ so that
  $$
  \vert \spl_3^1(\lri) : \Cen_{\spl_3^1(\lri)} (\bfb + p^{n+1}\lat)
  \vert = q^{-4} q^{4(n+1)} = q^{4n}.
  $$
  These three observations yield \eqref{equ:centr=q4n}, as wanted.

  To establish the last part of the induction claim, consider the
  relevance of the values of $y_2,y_3$ for lifting one of these
  matrices one step further.  Consider $\mathbf{b} =
  \mathbf{b}_c(y_1,y_2,y_3,y_4) = \mathbf{b}_{\tilde
    c}(y_1,y_2,y_3,0)$, where $\tilde c := c + p^{n-1} y_4 \in
  T_\lri(n+1)$, similarly as above, but allowing $y_1,\ldots,y_4 \in
  T_\lri(2)$.  Then the centraliser condition $[\mathbf{b},
  \mathbf{x}] \equiv 0$ modulo $p^{n+2} \lat$ leads, on the diagonal
  entries (cf.\ conditions (i) and (viii) in the list above), to the
  restriction
  \begin{align*}
    0 & \equiv_{p^{n+2}} \left( 4 p^2 {\tilde c}^2 x_2 + p^n (2y_1
      x_2 + y_2 x_3 + y_3 x_8) \right) - 2 \cdot \left( p^2 {\tilde c}^2 x_2
      - p^n ( y_1 x_2 - y_2 x_3 + 2 y_3 x_8) \right) \\
    & = p^n (3y_2 x_3 - 3 y_3 x_8)
  \end{align*}
  which is equivalent to $0 \equiv_p y_2 x_3 - y_3 x_8$, as $p=3$.
  Unless both $y_2$ and $y_3$ are congruent to $0$ modulo $p$, this
  leads to an unwanted restriction on $x_3$ and $x_8$.  This shows
  that it is enough to look for representatives of the elements of
  $\Menge^{[3]}_{n+1}$ within the sets $\mathbf{a}_{\tilde c} +
  \spl_3^{n+1}(\lri)$ modulo $\spl_3^{n+2}(\lri)$, where
  $\mathbf{a}_{\tilde c}$ arises from ${\tilde c} \equiv c$ modulo
  $p^{n-1}$.
\end{proof}

\subsubsection{} Next we consider elements belonging to cosets modulo
$p\lat$ of type \caseC{}.  Similarly as for type \caseB{}, we claim
that the relevant integral for type \caseC{} is equal to the
$\mathcal{Z}_\lri^{[2]}(r,t)$.  The latter may be computed easily from the
integer sequence $\Anzahl^{[2]}_n$, $n\in\N_0$, defined by
\begin{equation*}
  \Anzahl^{[2]}_n := \# \{ \bfy + (\mfp^{n+1})^{(8)} \mid \bfy \in
  \mfp^{(8)} \text{ such that } \lvert p \rvert_{\mfp} = 
  q^{-n-1}\}=
  \begin{cases}
    1 & \text{ if }n=0,\\
    0 & \text{ if }n \geq 1,
  \end{cases}
\end{equation*}
observing that the numbers $\Mass^{[2]}_l$ defined in
\eqref{equ:def_Mass^2} satisfy
\begin{equation*}
  \Mass^{[2]}_l = q^{-8l}\Anzahl^{[2]}_{l-1} - q^{-8(l+1)}\Anzahl^{[2]}_l.
\end{equation*}
The claim follows from the following proposition.
\begin{pro}
  For $n \in \N_0$ the set
  \begin{multline*}
    \Menge^{[2]}_n := \left\{ \mathbf{a} + p^{n+1}\lat \in
      \lat/p^{n+1}\lat \mid \mathbf{a} \equiv \left(
      \begin{smallmatrix} 0 & 1 & 0 \\ 0 & 0 & 1 \\ 0 & 0 & 0
      \end{smallmatrix} \right) \textnormal{ modulo $p\lat$} \right. \\
      \left. \textnormal{and } \lvert \spl_3(\lri) :
      \Cen_{\spl_3(\lri)} (\mathbf{a} + p^{n+1}\lat) \rvert =
      q^{4(n+1)} \right\}
  \end{multline*}
has cardinality $\Anzahl^{[2]}_n$.
\end{pro}

\begin{proof}
  The case $n=0$ is a simple computation.  The only candidate for an
  element of $\Menge^{[2]}_0$ is $\mathbf{a}_0 + p\lat$, where
  $\mathbf{a}_0 :=
  \left( \begin{smallmatrix} 0 & 1 & 0 \\
      0 & 0 & 1 \\ 0 & 0 & 0 \end{smallmatrix} \right)$, and a short
  computation reveals that
  $$
  \lvert \spl_3(\lri) : \Cen_{\spl_3(\lri)} (\mathbf{a}_0 + p\lat) \rvert =
  q^{4}.
  $$
  Indeed, for
  \begin{equation} \label{equ:generic_sl_matrix}
  \mathbf{x} =
    \begin{pmatrix}
      x_1 & x_2 & x_3 \\
      x_4 & x_5 & x_6 \\
      x_7 & x_8 & x_9
    \end{pmatrix}
    \in \spl_3(\lri)
  \end{equation}
  the commutator identity
  \begin{equation} \label{equ:comm_id} [ \mathbf{a}_0, \mathbf{x} ] =
    \begin{pmatrix}
      x_4 & x_5 - x_1 & x_6 - x_2 \\
      x_7 & x_8 - x_4 & x_9 - x_5 \\
      0 & -x_7 & -x_8
    \end{pmatrix}
  \end{equation}
  shows that $\mathbf{x} \in \Cen_{\spl_3(\lri)}(\mathbf{a}_0 +p\lat)$ if
  and only if the following congruences are satisfied:
  $$
  x_4 \equiv_p -x_8, \quad x_5 \equiv_p x_9 \equiv_p x_1, \quad x_6
  \equiv_p x_2, \quad x_7 \equiv_p 0,
  $$
  where $x_1, x_2, x_3, x_8$ can be chosen freely.  Thus $\lvert \Menge^{[2]}_0
  \rvert = 1$, as claimed.

  Next suppose that $n \geq 1$, and consider a lift $\mathbf{b} \in
  \mathbf{a}_0 + \spl_3^1(\lri)$ modulo $\spl_3^{2}(\lri)$.  Replacing
  $\mathbf{b}$ by a conjugate under $\GL_3^1(\lri)$ if necessary, we
  may assume, by a computation similar to~\eqref{equ:conjugation},
  that $\mathbf{b}$ is of the form
  $$
  \mathbf{b} = \mfb(y_1,y_2) = \mathbf{a}_0 +
  \begin{pmatrix}
    0 & 0 & 0 \\
    0 & 0 & 0 \\
    py_1 & py_2 & 0 
  \end{pmatrix} = 
  \begin{pmatrix}
    0 & 1 & 0 \\
    0 & 0 & 1 \\
    py_1 & py_2 & 0 
  \end{pmatrix},
  $$ where $y_1,y_2 \in T_\lri(1)$ are Teichm\"uller representatives
  in $\lri$.  Suppose that $\mathbf{x}$, as in
  \eqref{equ:generic_sl_matrix}, lies in
  $\Cen_{\spl_3(\lri)}(\mathbf{b} +p^2 \lat)$.  Then the commutator
  identity
  $$
  [ \mathbf{b}, \mathbf{x}] = [ \mathbf{a}_0, \mathbf{x}] +
  \begin{pmatrix}
    -py_1 x_3 & -py_2 x_3 & 0 \\
    -py_1 x_6 & -py_2 x_6 & 0 \\
    p(y_1(x_1 - x_9) + y_2 x_4) & p(y_1 x_2 + y_2(x_5 - x_9)) & p(y_1
    x_3 + y_2 x_6)
  \end{pmatrix},
  $$
  in conjunction with \eqref{equ:comm_id}, reveals that
  $$
  x_4 - py_1 x_3 \equiv_{p^2} (x_8 - x_4)-p y_2 x_6 \equiv_{p^2} - x_8
  + p(y_1 x_3 + y_2 x_6),
  $$
  hence $3x_4 \equiv_{p^2} 3x_8 \equiv_{p^2} 0$ irrespective of the
  particular values of $y_1,y_2$.  Furthermore, inspection of the
  $(1,2)$- and $(2,3)$-entries of the commutator identity shows that
  $x_1 \equiv_p x_5 \equiv_{p^2} x_9$.  Since $x_1 + x_5 + x_9
  \equiv_p 0$, this yields $x_1 \equiv_p x_5 \equiv_p x_9 \equiv_p 0$.
  As before, $x_4, x_5, x_6, x_7, x_9$ are determined by the values of
  $x_1, x_2, x_3, x_8$, but the latter satisfy the extra condition
  $x_1 \equiv_p x_8 \equiv_p 0$.  Hence the relevant index $\lvert
  \spl_3(\lri) : \Cen_{\spl_3(\lri)} (\mathbf{b} + p^2\lat) \rvert
  \geq q^{10}$.  This shows that $\Menge^{[2]}_1 = \varnothing$ and consequently
  $\Menge^{[2]}_n = \varnothing$ for all $n \geq 1$.
\end{proof}

\subsection{} \label{sec:summand_2} In order to conclude the
justification of the equations~\eqref{equ:S1_and_S2} we explain how
one obtains the summand $\mathcal{S}_2(r,t)$.  For this we decompose
$\lat \setminus \spl_3(\lri)$ into cosets modulo $p\lat$.  As $p=3$,
every element in this domain is of the form $u (p^{-1} \Id +
\mathbf{x})$, where $u \in \lri^*$ and $\mathbf{x} \in \gl_3(\lri)$
has trace $-1$.  We are thus led to decomposing the finite affine
space of matrices $\{ \mathbf{x} \in \gl_3(\mathbb{F}_q) \mid
\Tr(\mathbf{x}) = -1 \}$ into cosets modulo the $1$-dimensional
subspace $\mathbb{F}_q \overline{\mathbf{z}}$, spanned by the
reduction modulo $\mfp$ of $\mathbf{z} := \mathbf{h}_{12} -
\mathbf{h}_{23}$.  Of course, the latter coincides with the space of
scalar matrices over $\F_q$.  An overview of the orbits in $\{
\mathbf{x} \in \gl_3(\mathbb{F}_q) \mid \Tr(\mathbf{x}) = -1 \}$ under
the adjoint action of $\GL_3(\mathbb{F}_q)$ is provided in
Table~\ref{table_p=3_tr-1}; see Appendix~\ref{sec:aux_sl3_p=3} for a
short discussion.  The corresponding total number of cosets modulo
$\F_q \overline{\mathbf{z}}$, given in the last column of the table,
is obtained upon division by $q$.  Indeed, a short calculation
confirms that the sizes of the orbits listed in
Table~\ref{table_p=3_tr-1} add up to $q^7$ so that with $q-1$ choices
for $u$ modulo $\mfp$ we obtain $(q-1)q^7 = \lvert \lat : p\lat \rvert
- \lvert \spl_3(\lri) : p\lat \rvert$, as wanted.

\begin{table}[H]
  \centering
  \begin{tabular}{|c|l|l|l|l|}
    \hline
    type &  & no.\ of orbits & size of each orbit & total
    number modulo $\F_q \overline{\mathbf{z}}$ \\
    \hline\hline
    \caseG{} & irregular\ & $q$ & $(q^2+q+1)q^2$ &
    $(q^2+q+1)q^2$ \\ 
    \caseI{} & regular & $(q-3)q/6$ &
    $(q^2+q+1)(q+1)q^3$ & $(q-3)(q^2+q+1)(q+1)q^3/6$ \\
    \caseJ{} & regular & $(q-1)q/2$ &
    $(q^3-1)q^3$ & $(q^3-1)(q-1)q^3/2$ \\ 
    \caseK{}& regular & $q^2/3$ & $(q^2-1)(q-1)q^3$ &
    $(q^2-1)(q-1)q^4/3$ \\
    \caseH{}& regular & $q$ & $(q^3-1)(q+1)q^2$
    & $(q^3-1)(q+1)q^2$ 
    \\
    \hline
  \end{tabular}
  \caption{Adjoint orbits in $\{ \mathbf{x} \in \gl_3(\F_q) \mid
    \Tr(\mathbf{x})=-1 \}$ under the action of $\GL_3(\F_q)$, $q
    \equiv_3 0$}
  \label{table_p=3_tr-1}
\end{table}

That $\mathcal{Z}_\lri^{[0]}(r,t)$ is the correct integral for the
types \caseI{}, \caseJ{}, \caseK{}, \caseH{} is clear, because they
correspond to regular elements modulo $p$.  We need to link the
contributions by irregular elements belonging to cosets of type
\caseG{} to the integral $\mathcal{Z}_\lri^{[1]}(r,t)$, as shown
within the summand $\mathcal{S}_2(r,t)$ in \eqref{equ:S1_and_S2}.

Hence let us consider elements belonging to cosets modulo $p\lat$ of
type~\caseG{}.  A typical coset of this type is $\mathbf{a}_0 +
p\lat$, where $\mathbf{a}_0 := \left( \begin{smallmatrix} 0 & 0 & 0 \\
    0 & 0 & 0 \\ 0 & 0 & -1 \end{smallmatrix} \right)$, and each coset
has measure $\mu(p\lat) = q^{-7}$.  The determinant of the Jacobi
matrix associated to $\iota_0 : \lat \rightarrow
\Hom_\lri(\spl_3(\lri),\lri)$ contributes another factor $\lvert 3
\rvert_\mfp = q^{-1}$.  The integral over $\mfp \times (\mathbf{a} +
p\lat)$ with Jacobi factor $q^{-1}$ can thus be described as an
integral over $\mfp \times \mfp^{(8)}$. Similarly as in cases \caseC{}
and \caseB{}, we claim that this integral equals
$\mathcal{Z}_\lri^{[1]}(r,t)$. The latter may be computed from the
integer sequence $\Anzahl^{[1]}_n$, $n\in\N_0$, defined by
\begin{align*}
\Anzahl^{[1]}_n &:= \#\{\bfy + (\mfp^{n+1})^{(8)} \mid \bfy \in
\mfp^{(8)} \text{ such that }\lVert \{ y_1, y_2, y_3, p^{n+1} \}\rVert
_{\mfp}=q^{-n-1}\}\\ 
&= \#\{\bfy + (\mfp^{n+1})^{(8)} \mid \bfy \in
\mfp^{(8)} \text{ such that } y_1,y_2,y_3 \in \mfp^{n+1}\}\\
&= \vert
\mfp^{(5)}/(\mfp^{n+1})^{(5)} \vert\\
&= q^{5n},
\end{align*}
describing the lifting behaviour of points modulo $\mfp^{n+1}$ on the
variety defined by the integrand of
$\mathcal{Z}_\lri^{[1]}(r,t)$. Indeed, we observe that the numbers
$\Mass^{[1]}_l$ defined in \eqref{equ:def_Mass^1} satisfy
\begin{equation*}
\Mass^{[1]}_l = q^{-8l}\Anzahl^{[1]}_{l-1} -
q^{-8(l+1)}\Anzahl^{[1]}_l. 
\end{equation*}
The following proposition establishes the claim and thereby concludes
the overall proof of Theorem~\ref{thm:SL3_p=3}.

\begin{pro}
For $n \in \N_0$ the set
 \begin{multline*}
   \Menge^{[1]}_n := \left\{ \mathbf{a} + p^{n+1}\lat \in
     \lat/p^{n+1}\lat \mid \mathbf{a} \equiv p^{-1} \Id + \left(
     \begin{smallmatrix} 0 & 0 & 0 \\ 0 & 0 & 0 \\ 0 & 0 & -1
     \end{smallmatrix} \right) \textnormal{ modulo $p\lat$} \right. \\
     \left. \textnormal{and } \lvert \spl_3(\lri) :
     \Cen_{\spl_3(\lri)} (\mathbf{a} + p^{n+1}\lat) \rvert =
     q^{4(n+1)} \right\}
  \end{multline*}
has cardinality $\Anzahl^{[1]}_n$.
\end{pro}

\begin{proof}
  We argue by induction on $n$.  The case $n=0$ is a simple
  computation.  Indeed, the only candidate for an element of $\Menge^{[1]}_0$ is
  $\mathbf{a}_0 + p\lat$, where $\mathbf{a}_0 := p^{-1} \Id +
  \left( \begin{smallmatrix} 0 & 0 & 0 \\
      0 & 0 & 0 \\ 0 & 0 & -1 \end{smallmatrix} \right)$, and a short
  computation reveals that, indeed,
  $$
  \lvert \spl_3(\lri) : \Cen_{\spl_3(\lri)} (\mathbf{a}_0 + p\lat) \rvert =
  \lvert \spl_3(\F_q) : \Cen_{\spl_3(\F_q)} (\overline{\mathbf{a}_0})
  \rvert = q^{4},
  $$
  where $\overline{\mathbf{a}_0}$ denotes the image of $\mathbf{a}_0$
  in $\spl_3(\F_q)$; cf.\ \eqref{equ:centr_type_3}.  Thus $\lvert \Menge^{[1]}_0
  \rvert = 1$, as claimed.

  Now suppose that $n \geq 1$.  In fact, we will prove more than
  stated in the proposition.  For any $l \in \N$, let $T_\lri(l)$
  denote the representatives for $\lri / \mfp^l$ derived from the
  Teich\-m\"uller representatives for $\lri/\mfp$; see
  \eqref{equ:Teichmueller}.  By induction on $n$, the following
  assertions are proved below.

  \smallskip

  \noindent
  \emph{Claim.} Every matrix $\mathbf{a} \in \lat$ with $\mathbf{a} +
  p^{n+1}\lat \in \Menge^{[1]}_n$ can be conjugated by elements of
  $\GL_3^1(\lri)$ to the `normal' form
  $$
  p^{-1} \Id +
  \begin{pmatrix}
    c & 0 & 0 \\
    0 & c & 0 \\
    0 & 0 & -1 -2 c
  \end{pmatrix}
  \qquad \text{modulo $p^{n+1}\lat$,}
  $$ where $c \in T_\lri(n)$.  These matrices modulo $p^{n+1} \lat$
 form a complete set of representatives for the $\GL_3^1(\lri)$-orbits
 comprising $\Menge^{[1]}_n$.  Moreover, the index in $\GL_3^1(\lri)$
 of the centraliser of any such matrix modulo $p^{n+1}\lat$ is
 $q^{4n}$, and $\lvert \Menge^{[1]}_n \rvert = \lvert T_\lri(n) \rvert
 q^{4n} = q^{5n}$, as wanted.
  
  \smallskip

  Let $c \in T_\lri(n)$ and put
  $$
  \mathbf{a} := \bfa_c :=
  \begin{pmatrix}
    c & 0 & 0 \\
    0 & c & 0 \\
    0 & 0 & -1 - 2c
  \end{pmatrix}.
  $$ 
  Clearly, the eigenvalues of $\mathbf{a}$ are $c$ and $-1-2c$ with
  multiplicities $2$ and $1$ respectively.  Hence $c$ modulo $p^{n-1}$
  is an invariant of the $\GL_3^1(\lri)$-orbit of $\mathbf{a}$ modulo
  $p^n \lat$.  By induction, it suffices to look for representatives
  of the elements of $\Menge^{[1]}_n$ within the set $\mathbf{a} +
  \spl_3^n(\lri)$ modulo $p^{n+1} \lat$.  We consider the set
  $\mathbf{a} + \spl_3^n(\lri)$ modulo $\spl_3^{n+1}(\lri)$, up to
  conjugation by $\GL_3^n(\lri)$.  Let
  $$
  \mathbf{x} :=
  \begin{pmatrix}
    A & \underline{b}^{\text{t}} \\
    \underline{d} & a 
  \end{pmatrix}
  \in \Mat_3(\lri), \qquad \text{where $A \in \Mat_2(\lri)$ and
    $\underline{b}, \underline{d} \in \lri^{(2)}$, $a\in\lri$.}  
  $$
  Then, as $p=3$,
  \begin{align}
    \left[ \mathbf{a}, \mathbf{x} \right] & =
    \begin{pmatrix}
       0 & (1+3c) \underline{b}^{\text{t}} \\
      -(1+3c)\underline{d} & 0
    \end{pmatrix} \label{equ:commutator_complete_tr-1} \\
    & \equiv_{\mfp}
    \begin{pmatrix}
       0 & \underline{b}^{\text{t}} \\
      \underline{d} & 0
    \end{pmatrix}. \label{equ:commutator_congruence_tr-1}
  \end{align}
  As in the proof of Proposition~\ref{pro:schwierigster_fall}, the
  congruence \eqref{equ:commutator_congruence_tr-1} shows that the
  elements
  $$
  \mathbf{b}_c(Y) := \mathbf{a} + p^n \begin{pmatrix}
    Y & 0 \\
    0 & -\Tr(Y)
    \end{pmatrix},
  \quad \text{with $Y = \begin{pmatrix} y_1 & y_2 \\ y_3 & y_4
    \end{pmatrix} \in \Mat_2(T_\lri(1))$,} 
  $$ form a complete set of representatives for the
  $\GL_3^n(\lri)$-orbits of $\mathbf{a} + \spl_3^n(\lri)$ modulo
  $\spl_3^{n+1}(\lri)$, and with the extra restriction $\Tr(Y) =0$ a
  complete set of representatives modulo $p^{n+1} \lat$.

  Consider one of these lifts, $\mathbf{b} = \mathbf{b}_c(Y)$ and put
  $z := \Tr(Y)$.  In order to describe the centraliser index of
  $\mathbf{b} + p^{n+1}\lat$ in $\spl_3(\lri)$, we consider a generic
  matrix
  $$
  \qquad \mathbf{x} :=
  \begin{pmatrix}
    A & \underline{b}^{\text{t}} \\
    \underline{d} & -\Tr(A) 
  \end{pmatrix}
  \in \spl_3(\lri), \qquad \text{where $A \in \Mat_2(\lri)$ and
    $\underline{b}, \underline{d} \in \lri^{(2)}$}
  $$
  and compute
  $$
  [\mathbf{b},\mathbf{x}] = [\mathbf{a},\mathbf{x}] + p^n
  \begin{pmatrix}
     YA - AY & Y \underline{b}^{\text{t}} - z \underline{b}^{\text{t}} \\
     z \underline{d} - \underline{d} Y & 0
  \end{pmatrix}.
  $$
  Taking into account \eqref{equ:commutator_complete_tr-1}, the
  condition $[\mathbf{b}, \mathbf{x}] \equiv 0$ modulo $p^{n+1} \lat$ can
  be expressed in terms of the following list of restrictions on the
  entries of $\mathbf{x}$, involving as parameters the matrix $Y \in
  \Mat_2(T_\lri(1))$ and $z = \Tr(Y)$:
  \begin{enumerate}
  \item[(i)] $YA - AY \equiv_p 0$,
  \item[(ii)] $((1 + 3c - p^nz) \Id + p^n Y) \underline{b}^{\text{t}}
    \equiv_{p^{n+1}} 0$,
  \item[(iii)] $\underline{d} ((1 + 3c - p^nz) \Id + p^n Y)
    \equiv_{p^{n+1}} 0$.
  \end{enumerate}
  If these conditions are to hold then the last two congruences show
  that $\underline{b}$ and $\underline{d}$ are to be $0$ modulo
  $p^{n+1}$.  From this we observe that $\lvert \spl_3(\lri) :
  \Cen_{\spl_3(\lri)} (\mathbf{a} + p^{n+1}\lat) \rvert = q^{4(n+1)}$
  (and not larger) if and only if one can choose $A$ freely, i.e.\ if
  the first condition does not impose extra restrictions.  This
  implies that $Y$ is scalar, and $\Tr(Y) = 0$ implies that $Y=0$.
  The proof concludes in analogy to the proof of
  Proposition~\ref{pro:schwierigster_fall}
\end{proof}



\appendix

\section{Adjoint action of $\GL_2(\F_q)$ on $\spl_2(\F_q)$}
\label{sec:aux_sl2}

In Section~\ref{sec:full_sl_2}, we require an overview of the elements
in $\spl_2(\mathbb{F}_q)$ up to conjugacy under the group
$\GL_2(\F_q)$.  We distinguish four different types, labelled $0$,
$1$, $2$a,~$2$b.  The total number of elements of each type and the
isomorphism types of their centralisers in $\SL_2(\F_q)$ are
summarised in Tables~\ref{table_sl2_1} and \ref{table_sl2_2}.  We
briefly discuss the four different types.

Type $0$ consists of the zero matrix, which does not feature in our
calculation but is shown for completeness.  Its centraliser is the
entire group $\SL_2(\F_q)$.

Type $1$ consists of nilpotent matrices with minimal polynomial equal
to $X^2$ over $\F_q$.  The centraliser of a typical element is
$$
\Cen_{\SL_2(\F_q)} \left( \Bigl(\begin{smallmatrix} 0 & 1 \\ 0 & 0
  \end{smallmatrix} \Bigr) \right) = \left\{ \Bigl(
  \begin{smallmatrix} a & b \\ 0 & a \end{smallmatrix} \Bigr) \mid a,b
  \in \F_q, a^2 = 1 \right\}
$$
and matrices of type $1$ are regular.

Type $2$a consists of semisimple matrices with distinct eigenvalues
$\lambda, -\lambda \in \F_q \setminus \{0\}$.  The minimal polynomial
of such elements over $\F_q$ is equal to $X^2-\lambda^2$.  The
centraliser of a typical element is
$$
\Cen_{\SL_2(\F_q)} \left( \Bigl(\begin{smallmatrix} \lambda & 0 \\ 0 &
    -\lambda \end{smallmatrix} \Bigr) \right) = \left\{ \Bigl(
  \begin{smallmatrix} a & 0 \\ 0 & b \end{smallmatrix} \Bigr) \mid a,b
  \in \F_q, ab = 1 \right\}
$$
and matrices of type $2$a are regular.

Type $2$b consists of semisimple matrices with eigenvalues $\lambda,
\lambda^q \in \F_{q^2} \setminus \F_q$.  The minimal polynomial of
such elements over $\F_q$ is equal to $(X-\lambda)(X-\lambda^q)$.  The
centraliser of a typical element is isomorphic to the group of
elements of norm $1$ in the field $\F_{q^2}$ and matrices of type $2$b
are regular.


\section{Adjoint action of $\GL_3(\F_q)$ on $\spl_3(\F_q)$}
\label{sec:aux_sl3}

This appendix is almost identical to Appendix~B in \cite{AvKlOnVo10+}
and included for the reader's convenience.  Let $\F_q$ be a finite
field of characteristic not equal to $3$.  We give an overview of the
elements in $\spl_3(\F_q)$ up to conjugacy under the group
$\GL_3(\F_q)$.  For this we distinguish eight different types,
labelled $0$, $1$, $2$, $3$, $4$a, $4$b, $4$c,~$5$.  The total number
of elements of each type and the isomorphism types of their
centralisers in $\SL_3(\F_q)$ are summarised in Tables~7.1 and 7.2
in~\cite{AvKlOnVo10+}.  We briefly discuss the eight different types.

Type $0$ consists of the zero matrix, which does not feature in our
calculation but is shown for completeness.  Its centraliser is the
entire group $\SL_3(\F_q)$.

Type $1$ consists of nilpotent matrices with minimal polynomial equal
to $X^3$ over~$\F_q$.  The centraliser of a typical element is
\begin{equation}\label{equ:centr_type_1}
  \Cen_{\SL_3(\F_q)} \left( \Bigl(\begin{smallmatrix} 0 & 1 & 0 \\ 0 & 0 & 1
      \\ 0 & 0 & 0 \end{smallmatrix} \Bigr) \right) = \left\{
    \Bigl( \begin{smallmatrix} a & b & c \\ 0 & a & b \\ 0 & 0 &
      a \end{smallmatrix} \Bigr) \in \GL_3(\F_q) \mid  a^3 = 1 \right\}
\end{equation}
and matrices of type $1$ are regular.

Type $2$ consists of nilpotent matrices with minimal polynomial equal
to $X^2$ over~$\F_q$.  The centraliser of a typical element is
\begin{equation}\label{equ:centr_type_2}
  \Cen_{\SL_3(\F_q)} \left( \Bigl(\begin{smallmatrix} 0 & 1 & 0 \\ 0 & 0 & 0
      \\ 0 & 0 & 0 \end{smallmatrix} \Bigr) \right) = \left\{
    \Bigl( \begin{smallmatrix} a & b & c \\ 0 & a & 0 \\ 0 & d &
      e \end{smallmatrix} \Bigr) \in \GL_3(\F_q) \mid a^2 e = 1 \right\}
\end{equation}
and matrices of type $2$ are irregular.

Type $3$ consists of semisimple matrices with eigenvalues $\lambda \in
\F_q \setminus \{0\}$ of multiplicity $2$ and $\mu := -2\lambda$.  The
minimal polynomial of such elements over $\F_q$ is equal to
$(X-\lambda)(X-\mu)$.  The centraliser of a typical element is
\begin{equation}\label{equ:centr_type_3}
  \Cen_{\SL_3(\F_q)} \left( \Bigl(\begin{smallmatrix} \lambda & 0 & 0 \\ 0 &
      \lambda & 0 \\ 0 & 0 & \mu \end{smallmatrix} \Bigr) \right) =
  \left\{ \Bigl( \begin{smallmatrix} a & b & 0 \\ c & d & 0 \\ 0 & 0 &
      e \end{smallmatrix} \Bigr) \in \GL_3(\F_q) \mid (ad-bc) e = 1
  \right\}
\end{equation}
and matrices of type $3$ are irregular.

Type $4$a consists of semisimple matrices with distinct eigenvalues
$\lambda,\mu,\nu := -\lambda-\mu \in \F_q \setminus \{0\}$.  The
minimal polynomial of such elements over $\F_q$ is equal to
$(X-\lambda)(X-\mu)(X-\nu)$.  The centraliser of a typical element is
\begin{equation}\label{equ:centr_type_4a}
\Cen_{\SL_3(\F_q)} \left( \Bigl(\begin{smallmatrix} \lambda & 0 & 0 \\
    0 & \mu & 0 \\ 0 & 0 & \nu \end{smallmatrix} \Bigr) \right) =
\left\{ \Bigl( \begin{smallmatrix} a & 0 & 0 \\ 0 & b & 0 \\ 0 & 0 &
    c \end{smallmatrix} \Bigr) \in \GL_3(\F_q) \mid abc = 1 \right\}
\end{equation}
and matrices of type $4$a are regular.

Type $4$b consists of semisimple matrices with eigenvalues
$\lambda,\mu := \lambda^q \in \F_{q^2} \setminus \F_q$ and $\nu :=
-\lambda-\mu \in \F_q$.  The minimal polynomial of such elements over
$\F_q$ is equal to $(X-\lambda)(X-\mu)(X-\nu)$.  The centraliser of a
typical element is isomorphic to the multiplicative group of the field
$\F_{q^2}$ and matrices of type $4$b are regular.

Type $4$c consists of semisimple matrices with eigenvalues
$\lambda,\mu := \lambda^q, \nu := \lambda^{q^2} \in \F_{q^3} \setminus
\F_q$ with $\lambda + \mu + \nu = 0$.  The minimal polynomial of such
elements over $\F_q$ is equal to $(X-\lambda)(X-\mu)(X-\nu)$.  The
centraliser of a typical element is isomorphic to the group of
elements of norm $1$ in the field $\F_{q^3}$ and matrices of type $4$c
are regular.

Type $5$ consists of matrices with eigenvalues $\lambda \in \F_q
\setminus \{0\}$ of multiplicity $2$ and $\mu := -2\lambda$.  The
minimal polynomial of such elements over $\F_q$ is equal to
$(X-\lambda)^2(X-\mu)$.  The centraliser of a typical element is
\begin{equation}\label{equ:centr_type_5}
  \Cen_{\SL_3(\F_q)} \left( \Bigl(\begin{smallmatrix} \lambda & 1 & 0 \\ 0 &
      \lambda & 0 \\ 0 & 0 & \mu \end{smallmatrix} \Bigr) \right) =
  \left\{ \Bigl( \begin{smallmatrix} a & b & 0 \\ 0 & a & 0 \\ 0 & 0 &
      c \end{smallmatrix} \Bigr) \in \GL_3(\F_q) \mid a^2 c = 1 \right\}
\end{equation}
and matrices of type $5$ are regular.

\section{Auxiliary results regarding $\spl_3(\F_q)$ and $\gl_3(\F_q)$
  for $q \equiv_3 0$}
\label{sec:aux_sl3_p=3}

\subsection{} Let $\F_q$ be a finite field of characteristic equal to
$3$.  In Section~\ref{sec:sl3_Z3}, we require an overview of the
elements in $\spl_3(\F_q)$ up to conjugacy under the group
$\GL_3(\F_q)$.  We distinguish six different types, labelled \caseA{},
\caseC{}, \caseB{}, \caseF{}, \caseD{} and \caseE{}, corresponding to
the types $0$, $1$, $2$, $4$a, $4$b and $4$c in the generic case
($p\neq 3$); cf.~Appendix~\ref{sec:aux_sl3}. There are no analogues to
types $3$ and $5$ in characteristic $3$.  The total number of elements
of each type, number of orbits and orbit sizes, are summarised in
Table~\ref{table_p=3}.  We briefly discuss the six different types.

Type \caseA{} (corresponding to type $0$ in the generic case) consists
of the scalar matrices.  This type is listed for completeness and does
not feature in our calculation.

Type \caseC{} (corresponding to type $1$) consists of matrices with
minimal polynomial equal to $(X-\lambda)^3$ over~$\F_q$.  There are
$q$ possible values for $\lambda \in \F_q$, hence $q$ orbits.  The
orbit sizes are as in the generic case; see \eqref{equ:centr_type_1}.
In contrast to the generic case, matrices of type \caseC{} are
irregular in characteristic $3$.

Type \caseB{} (corresponding to type $2$) consists of matrices with
minimal polynomial equal to $(X-\lambda)^2$ over~$\F_q$.  There are
$q$ possible values for $\lambda \in \F_q$, hence $q$ orbits.  The
orbit sizes are as in the generic case; see \eqref{equ:centr_type_2}.
In contrast to the generic case, matrices of type \caseC{} are
irregular in characteristic $3$.

Type \caseF{} (corresponding to type $4$a) consists of matrices with
minimal polynomial equal to $(X-\lambda)(X-\mu)(X-\nu)$ over~$\F_q$,
with distinct $\lambda, \mu, \nu$ such that $\lambda + \mu + \nu = 0$.
There are $(q-1)q/6$ possible choices for $\{\lambda, \mu, \nu\}$,
hence the same number of orbits.  The orbit sizes are as in the
generic case; see \eqref{equ:centr_type_4a}.  As in the generic case,
matrices of type \caseF{} are regular.

Type \caseD{} (corresponding to type $4$b) consists of semisimple
matrices with eigenvalues $\lambda,\mu := \lambda^q \in \F_{q^2}
\setminus \F_q$ and $\nu := -\lambda-\mu \in \F_q$.  The number of
orbits and the orbit sizes are exactly as in the generic case; see
Appendix~\ref{sec:aux_sl3}.  Matrices of type \caseD{} are regular.

Type \caseE{} (corresponding to type $4$c) consists of semisimple
matrices with eigenvalues $\lambda,\mu := \lambda^q, \nu :=
\lambda^{q^2} \in \F_{q^3} \setminus \F_q$ with $\lambda + \mu + \nu =
0$.  In characteristic $3$, the number of elements in $\F_{q^3}
\setminus \F_q$ with trace $0$ in $\F_q$ is $q^2-q$.  Thus there are
$(q-1)q/3$ orbits.  The orbit sizes are as in the generic case; see
Appendix~\ref{sec:aux_sl3}.  Matrices of type \caseE{} are regular.

\subsection{} Let $\F_q$ be a finite field of characteristic equal to
$3$.  In Section~\ref{sec:sl3_Z3}, we also require an overview of the
elements in $\{ \mathbf{x} + \F_q \overline{\mathbf{z}} \mid
\mathbf{x} \in \gl_3(\F_q) \text{ with } \Tr(\mathbf{x})=-1 \}$ up to
conjugacy under the group $\GL_3(\F_q)$.  We distinguish five
different types, labelled \caseG{}, \caseI{}, \caseJ{}, \caseK{} and
\caseH{}, corresponding to the types $3$, $4$a, $4$b, $4$c and $5$ in
the generic case ($p\neq3$); cf.~Appendix~\ref{sec:aux_sl3}. There are
no analogues to types 1 and 2 in characteristic~$3$. The total number
of elements of each type, number of orbits and orbit sizes, are
summarised in Table~\ref{table_p=3_tr-1}.  We briefly discuss the five
different types.

Type \caseG{} (analogous to type $3$ in the generic case) consists of
semisimple matrices with eigenvalues $\lambda \in \F_q$ of
multiplicity $2$ and $\mu := \lambda-1$.  The minimal polynomial of
such an element over $\F_q$ is equal to $(X-\lambda) (X-\mu)$.  There
are $q$ possible values for $\lambda$, hence $q$ orbits.  The orbit
sizes are as in the generic case; see \eqref{equ:centr_type_3}.  As in
the generic case, matrices of type \caseG{} are irregular.

Type \caseI{} (analogous to type $4$a) consists of matrices with
minimal polynomial equal to $(X-\lambda)(X-\mu)(X-\nu)$ over~$\F_q$,
with distinct $\lambda, \mu, \nu$ such that $\lambda + \mu + \nu =
-1$.  One checks that there are $(q-3)q/6$ possible choices for
$\{\lambda, \mu, \nu\}$, hence the same number of orbits.  The orbit
sizes are as in the generic case; see \eqref{equ:centr_type_4a}.  As
in the generic case, matrices of type \caseI{} are regular.

Type \caseJ{} (analogous to type $4$b) consists of semisimple matrices
with eigenvalues $\lambda,\mu := \lambda^q \in \F_{q^2} \setminus
\F_q$ and $\nu := -\lambda-\mu-1 \in \F_q$.  The number of orbits and
the orbit sizes are exactly as in the generic case; see
Appendix~\ref{sec:aux_sl3}.  Matrices of type \caseJ{} are regular.

Type \caseK{} (analogous to type $4$c) consists of semisimple matrices
with eigenvalues $\lambda,\mu := \lambda^q, \nu := \lambda^{q^2} \in
\F_{q^3} \setminus \F_q$ with $\lambda + \mu + \nu = -1$.  In
characteristic $3$, the number of elements in $\F_{q^3} \setminus
\F_q$ with trace $-1$ in $\F_q$ is $q^2$.  Thus there are $q^2/3$
orbits.  The orbit sizes are as in the generic case; see
Appendix~\ref{sec:aux_sl3}.  Matrices of type \caseK{} are regular.

Type \caseH{} (analogous to type $5$) consists of matrices with
eigenvalues $\lambda \in \F_q$ of multiplicity $2$ and $\mu :=
\lambda-1$.  The minimal polynomial of such an element over $\F_q$ is
equal to $(X-\lambda)^2 (X-\mu)$.  There are $q$ possible values for
$\lambda$, hence $q$ orbits.  The orbit sizes are as in the generic
case; see \eqref{equ:centr_type_5}.  As in the generic case, matrices
of type \caseH{} are irregular.
 

\bigskip

\begin{acknowledgements}
 The authors would like to thank Alexander Lubotzky as well as the
 following institutions: the Batsheva de Rothschild Fund for the
 Advancement of Science, the EPSRC, the Mathematisches
 Forschungsinstitut Oberwolfach, the National Science Foundation and
 the Nuffield Foundation.
\end{acknowledgements}


\end{document}